\documentclass[10pt,reqno]{amsart}
\usepackage[applemac]{inputenc}
\usepackage[english]{babel}
\usepackage{caption}
\usepackage{bm}
\usepackage{bbm}
\usepackage{xcolor}
\usepackage{graphicx}
\usepackage{amssymb}
\usepackage{bbold}

\usepackage{color}

 \newtheorem{thm}{Theorem}[section]
 
 \newtheorem{lem}[thm]{Lemma}
 \newtheorem{prop}[thm]{Proposition}
 \newtheorem{rem}{Remark}
  \newtheorem*{rem*}{Remark}
 \numberwithin{equation}{section}
 
 \newcommand\eps{\varepsilon}

 \newcommand{\be}[1]{\begin{equation}\label{#1}}
\newcommand{\ee}{\end{equation}}

\def\squarebox#1{\hbox to #1{\hfill\vbox to #1{\vfill}}}

\begin{document}

\title[Symmetric ground states for a  relativistic model for nucleons ]{Symmetric ground states for a stationary relativistic mean-field model for nucleons in the nonrelativistic limit}

\author[M.J. Esteban]{Maria J. Esteban$^1$}
\address{$^1$Ceremade, Universit\'e Paris-Dauphine, Place de Lattre de Tassigny, F-75775 Paris C\'edex 16, France}
\email{esteban@ceremade.dauphine.fr}

\author[S. Rota Nodari]{Simona Rota Nodari$^{2,3}$}
\address{$^2$ CNRS, UMR 7598, Laboratoire Jacques-Louis Lions, F-75005, Paris, France}
\address{
$^3$ UPMC Univ Paris 06, UMR 7598, Laboratoire Jacques-Louis Lions, F-75005, Paris, France}
\email{rotanodari@ann.jussieu.fr}


\date{\today}

\begin{abstract}
In this paper we consider a model for a nucleon interacting with the $\omega$ and $\sigma$ mesons in the atomic nucleus. The model is relativistic, but we study it in the nuclear physics nonrelativistic limit, which is of a very different nature from the one of the atomic physics. Ground states with a given angular momentum are shown to exist for a large class of values for the coupling constants and the mesons' masses. Moreover, we show that, for a good choice of parameters, the very striking shapes of mesonic densities inside and outside the nucleus are well described by the solutions of our model.
\end{abstract}

\maketitle

\section{Introduction}

This article is concerned with the existence of ground states for a stationary relativistic mean-field model for atomic nuclei in the nuclear physics nonrelativistic limit. Moreover, a discussion about some interesting qualitative properties of the solutions of the model considered is provided.

Classically, at low energies, nuclear structure is described by quantum mechanical many-body problems of fermions interacting by a nonrelativistic interaction (see for instance \cite{Ring-Schuck}), and as explained in \cite{ring}, this interaction is understood to have its origin in the exchange of mesons between the bare nucleons.  If one introduces relativity into the model, then one has to propose a relativistic Lagrangian describing the point-interaction between the nuclei. But this is quite complicated. One of the chosen alternatives consists in introducing density dependent energy functionals for variational calculations of the Hartree-Fock type. The exchange terms are described by {\it ad-hoc}  ans\"atze, whose parameters are fitted to experimental data. The classical models in this framework are the Gogny and the Skyrme models. These models are nonrelativistic.

Another existing proposition concerns the so-called RMFT (\textit {Relativistic mean-field theory}). The relativistic mean field model is formulated on the basis of two approximations: the mean field approximation and the no-sea. On the one hand, thanks to the mean-field approximation, the fields for the mesons and the photons are treated as classical fields and the nucleons behave as noninteracting particles moving in these mean fields. On the other hand, thanks to the no-sea approximation, negative energy states belonging to the Dirac sea are not considered and the vacuum polarization is neglected.

The relativistic mean field theory is an effective theory: the Lagrangian of the model is an effective Lagrangian with respect to the mean-field and no-sea approximations. Since the effective Lagrangian is not derived rigorously from the no-approximation Lagrangian, the  parameters of the model must be adjusted on experimental data. Therefore, the effects of vacuum polarization as well as the correlation effects are not completely neglected but they are taken into account implicitly through the adjustment of model parameters.

During the last years, the relativistic mean-field theory has received a wide attention due to its successful description of many nuclear phenomena.
It has been shown that the relativistic mean-field model describes successfully the structure of the nucleus and provides a natural explanation for some relativistic effects observed experimentally such as the spin-orbit force. This is why the relativistic mean-field model can be viewed as the relativistic generalization of some nonrelativistic models, such as the Hartree-Fock model with Skyrme or Gogny interaction, where the effective forces, which are not appropriate in a relativistic formulation, are replaced by average potentials representing independent degrees of freedom.

For a more detailed description of the RMFT models, see \cite{ring}. The article \cite{meng} also contains an interesting discussion about these models and their numerical treatment.

Although often used in practice, the models of nuclear physics have rarely been discussed in the  mathematical community. Some nonrelativistic models of nuclear physics (of Hartree-Fock type) have been studied by D. Gogny and P.L. Lions in an article in 1986 (\cite{gognylions}). To  
our knowledge, S. Rota Nodari's paper \cite{rotanodarirmf} contains the first mathematical study of a model from relativistic nuclear physics.

In the RMFT model studied in this paper, we take into consideration the potentials created by the mesons $\sigma$ and $\omega$, defining a medium range attractive interaction and a short range repulsive interaction respectively. We neglect the meson $\rho$  describing the effects depending on the isospin\footnote{Isospin (contraction of isotopic spin) is a quantum number related to the strong interaction. Isospin was introduced by Heisenberg in 1932; he observed that the neutron is almost identical to the proton, apart from the fact that it carries no charge. In particular, their masses are close and they are indistinguishable under the strong interactions. So, the proton and the neutron appear to be two states of the same particle, the nucleon, associated with different isospin projections}, and we omit the photons which are related to the electromagnetic interaction (see \cite{reinhard}, \cite{greinermaruhn}). This static model which is currently known as the 
$\sigma$-$\omega$ model (\cite{waleckasigmaomega}, \cite{walecka}), is given by 
\begin{eqnarray}\label{eqdiracc}
\left[-ic\bm{\alpha}\cdot\nabla+\beta(mc^2+S)+V\right]\psi_j&=&(mc^2-\mu_j)\psi_j\,,\\
\label{eqsigmac}
\left[-\Delta+m_\sigma^2c^2\right]S&=&-g_\sigma^2 c \rho_s\,,\\
\label{eqomegac}
\left[-\Delta+m_\omega^2c^2\right]V&=&g_\omega^2 c \rho_0\,,
\end{eqnarray}
for $j=1,\ldots,A$ with $A$ the number of nucleons. In these equations the functions $\psi_j$ are the wave functions of the nucleons,  $c>0$ is the speed of light, $m$ is the mass of the nucleons, $\beta=\left(\begin{array}{cc}\mathbb{1}& 0\\0&-\mathbb{1}\end{array}\right)$, $\alpha_k=\left(\begin{array}{cc}0& \sigma_k\\\sigma_k&0\end{array}\right)$ $(k=1,2,3)$ and the $\sigma_k$ are the well known Pauli matrices.
The free Dirac operator $$H_c=-ic\bm{\alpha}\cdot\nabla+\beta mc^2$$ acts on $4$-spinors, i.e. functions from $\mathbb{R}^3$ to $\mathbb{C}^4$ and it is self-adjoint in $L^2(\mathbb{R}^3,\mathbb{C}^4)$, with domain $H^1(\mathbb{R}^3,\mathbb{C}^4)$ and form-domain $H^{1/2}(\mathbb{R}^3,\mathbb{C}^4)$. Moreover, $m_\sigma$ and $m_\omega$ are the masses of the $\sigma$ and $\omega$ meson respectively, and $g_\sigma$ and $g_\omega$ are coupling constants. The densities $\rho_s$ and $\rho_0$ are defined by
\begin{align*}
\rho_s=&\sum_{j=1}^{A}\bar{\psi_j}\psi_j:=\sum_{j=1}^{A}\psi_j^*\beta\psi_j,\\
\rho_0=&\sum_{j=1}^{A}\psi_j^*\psi_j.
\end{align*}

In this paper we are interested in the study of the nonrelativistic limit 
 of the above system. In order to do so,  we are going to pass to the limit as $c\to +\infty$ in a certain regime for the parameters and we are going to suppose that the ``energies" of the nucleons are small compared to the speed of light, that is,
 $0\le\mu_j\ll mc^2$. 

Writing $\psi_j=\left(\begin{array}{c}\varphi_j\\ \chi_j\end{array}\right)$, $\varphi_j, \chi_j:\mathbb{R}^3\to \mathbb{C}^2$, the densities $\rho_s$ and $\rho_0$ take the form
\begin{align*}
\rho_s=&\sum_{j=1}^{A}\left(|\varphi_j|^2-|\chi_j|^2\right)\,,\\
\rho_0=&\sum_{j=1}^{A}\left(|\varphi_j|^2+|\chi_j|^2\right)\,,
\end{align*}
and the equation (\ref{eqdiracc}) becomes
\begin{equation}\label{eveqs}
\left\{
\begin{aligned}
&-ic\bm{\sigma}\cdot\nabla\chi_j+(S+V)\varphi_j=-\mu_j\varphi_j\,,\\
&-ic\bm{\sigma}\cdot\nabla\varphi_j-(2mc^2+S-V-\mu_j)\chi_j=0\,,
\end{aligned}
\right.
\end{equation}
with $\bm{\sigma}=(\sigma_1,\sigma_2,\sigma_3)$.

Supposing that the quantity  $2mc^2+S-V-\mu_j$ is never equal to $0$, from the system (\ref{eveqs}), we formally obtain
\begin{equation}\label{eqdiraccdownf}
\chi_j=\frac{-ic\bm{\sigma}\cdot\nabla\varphi_j}{2mc^2+S-V-\mu_j}.
\end{equation}

To obtain the nonrelativistic limit of the system \eqref{eveqs}, it would be useful to expand the denominator in (\ref{eqdiraccdownf}) w.r.t. a small parameter when $c$ tends towards infinity. In atomic physics, the usual procedure consists in considering that  the quantity $\frac{S-V-\mu_j}{2mc^2}$ is small when $c\to+\infty$. In nuclear physics this approximation is not acceptable, because the quantity $S-V$ is large. Indeed, for $c\to+\infty$, one can write
\begin{align*}
S=&-\frac{1}{c}\left(\frac{g_\sigma}{m_\sigma}\right)^2\left[\frac{-\Delta}{m_\sigma^2c^2}+1\right]^{-1}\rho_s=-\frac{1}{c}\left(\frac{g_\sigma}{m_\sigma}\right)^2\rho_s+O\left(\frac{1}{c^3}\right)\,,\\
V=&\frac{1}{c}\left(\frac{g_\omega}{m_\omega}\right)^2\left[\frac{-\Delta}{m_\omega^2c^2}+1\right]^{-1}\rho_0=\frac{1}{c}\left(\frac{g_\omega}{m_\omega}\right)^2\rho_0+O\left(\frac{1}{c^3}\right).
\end{align*}
and, taking into account the physical values of the mesons' masses and of the coupling constants (see \cite{reinhard},\cite{ring}), we can suppose
\begin{align*}
\left(\frac{g_\sigma}{m_\sigma}\right)^2=&\left(\frac{g_\omega}{m_\omega}\right)^2+\lambda c\,,\\
\frac{1}{c}\left(\frac{g_\sigma}{m_\sigma}\right)^2=&\,\,\theta mc^2\,,
\end{align*} 
with $\lambda>0$ small and $\theta>0$. As a consequence, for $c\to+\infty$
\begin{align*}
S+V=& \,\,2\theta mc^2\sum_{j=1}^{A}|\chi_j|^2-\lambda\rho_0+o(1)\,,\\
S-V=& -2\theta mc^2\sum_{j=1}^{A}|\varphi_j|^2+\lambda\rho_0+o(1)\,,
\end{align*}
and so, from \eqref{eveqs},
$$
\left\{
\begin{aligned}
&-ic\bm{\sigma}\cdot\nabla\chi_k+\left(2\theta mc^2\sum_{j=1}^{A}|\chi_j|^2-\lambda\rho_0\right)\varphi_k+\mu_k\varphi_k+o(1)=0,\\
&-ic\bm{\sigma}\cdot\nabla\varphi_k-\left(2mc^2-2\theta mc^2\sum_{j=1}^{A}|\varphi_j|^2+\lambda\rho_0-\mu_k\right)\chi_k+o(1)=0.
\end{aligned}
\right.
$$
Next, if we suppose that $\frac{\lambda\rho_0}{c^2}=o(1)_{c\to+\infty}$, we obtain the limiting system
\be{limiteveqs}
\left\{
\begin{aligned}
&-ic\bm{\sigma}\cdot\nabla\chi_k+\left(2\theta mc^2\sum_{j=1}^{A}|\chi_j|^2-\lambda\rho_0\right)\varphi_k+\mu_k\varphi_k+o(1)_{c\to+\infty}=0,\\
&-i\bm{\sigma}\cdot\nabla\varphi_k-2mc\left(1-\theta \sum_{j=1}^{A}|\varphi_j|^2{+o(1)_{c\to+\infty}}\right)\chi_k+o(1)_{c\to+\infty}=0.
\end{aligned}
\right.
\ee
Finally, using the following rescaling
\begin{align*}
\tilde\varphi_j=&\sqrt{\theta}\varphi_j,\\
\tilde\chi_j=&-2mc\sqrt{\theta}\chi_j,
\end{align*}
and taking $c$ that goes to $+\infty$,
we get 
\begin{equation*}
\left\{
\begin{aligned}
&i\bm{\sigma}\cdot\nabla\tilde\chi_k+\sum_{j=1}^{A}|\tilde\chi_j|^2\tilde\varphi_k-\frac{2m\lambda}{\theta}\sum_{j=1}^{A}|\tilde\varphi_j|^2\tilde\varphi_k+2m\,\mu_k\,\tilde\varphi_k=0,\\
&-i\bm{\sigma}\cdot\nabla\tilde\varphi_k+\left(1-\sum_{j=1}^{A}|\tilde\varphi_j|^2\right)\tilde\chi_k=0.
\end{aligned}
\right.
\end{equation*}
In this paper we study the above system in the particular case when $A=1$, that is,  the case of a single nucleon. This is of course a very particular case, but it is the first step in the study of the multiple particle system which is much more involved and that will be the object of a forthcoming paper.
 The system that we study is the following
\begin{equation}\label{eqdiracnrl}
\left\{
\begin{aligned}
&i\bm{\sigma}\cdot\nabla\chi+|\chi|^2\varphi-a|\varphi|^2\varphi+b\varphi=0,\\
&-i\bm{\sigma}\cdot\nabla\varphi+\left(1-|\varphi|^2\right)\chi=0,
\end{aligned}
\right.
\end{equation}
with $a=\frac{2m\lambda}{\theta}$, $b=2m\mu\,$ and $\,\psi=\left(\begin{aligned}\varphi\\\chi\end{aligned}\right)$. 

Moreover, we are going to consider a particular ansatz which is classical when studying electronic or nucleonic wave functions. It is separable in spherical coordinates and consists in eigenfunctions of the spin-orbit (angular momentum) operator with corresponding eigenvalue equal to $-1$ (see \cite{Thaller}). We will look for solutions of (\ref{eqdiracnrl}) in the particular  form 
\begin{equation}\label{eqsol}
\psi(x)=\left(
\begin{aligned}
&g(r)\left(\begin{aligned}1\\0\end{aligned}\right)\\
&i f(r)\left(\begin{aligned}&\cos\vartheta\\&\sin\vartheta e^{i\phi}\end{aligned}\right)
\end{aligned}
\right)\,,
\end{equation}
where $f$ and $g$ are real valued radial functions. 
Then the equations for $f$ and $g$ are 
\begin{equation}\label{eqrad}
\left\{\begin{aligned}
f'+\frac{2}{r}f&=g(f^2-a g^2+b)\,,\\
g'&=f(1-g^2)\,,
\end{aligned}\right.
\end{equation}
where we assume $f(0)=0$ in order to avoid solutions with singularities at the origin. For any given value of $g(0)$ there is a local solution of (\ref{eqrad}).
Because of the fact that we want the corresponding solution $\psi$ of (\ref{eqdiracnrl}) to be square integrable, we consider only the solutions of (\ref{eqrad}) which go to $(0,0)$ as $r\to+\infty$. 

About this system our main result is the following

\begin{thm}\label{thmexistence} Given $a,b>0$ such that $a-2b>0$, there exists a solution $(f,g)\in\mathcal{C}^1([0,+\infty),\mathbb{R}^2)$ of the system \eqref{eqrad}
such that $f(0)=0$, and there exists a constant $C$ such that for $r>0$
$$
0<-f(r), g(r)\le C\exp(-K_{a,b}r)\,,
$$
with $K_{a,b}=\min\left\{\frac{b}{2},\frac{2a-b}{2a}\right\}$.
\end{thm}

This theorem and its proof have the same flavor as the main results and proofs in \cite{cazenavevazquez}. In that paper, Cazenave and Vazquez study solutions of the so-called Soler model, which consists in a nonlinear Dirac equation. They also consider solutions separable in spherical coordinates, with the same angular momentum constraint as we do. The main difference between their methods of proof and ours are related to the fact that, as we see below,  in our case there is a boundedness constraint for the initial value of $g$. This creates additional difficulties and another strategy is necessary for the proof.

 The study of the general tridimensional case, without ansatz on the function $\psi$, involves a completely different method of proof and will be done in a separate paper.

This paper is organized as follows. In Section \ref{conditionsparameters} we discuss the possible values of the parameters $a$ and $b$ for finite energy solutions of \eqref{eqrad} to exist. Moreover, some preliminary results about the system \eqref{eqrad} are proved there. Section \ref{existence} contains the proof of our main result, Theorem \ref{thmexistence}. Finally, Section \ref{numerics} is devoted to the presentation of some numerical computations which show the kind of physically interesting qualitative properties enjoyed by the solutions obtained in Section \ref{existence}. Those properties will be discussed and compared with existing experimental data.

\section{Conditions on the parameters and auxiliary results}

 \begin{prop}\label{conditionsparameters}
There is no  nontrivial solution of \eqref{eqrad} such that 
\be{limiting}(f(r), g(r)) \longrightarrow (0,0)\quad\mbox{as}\quad  r\longrightarrow +\infty
\,,
\ee
 unless $a-2b>0$. Moreover, for all the solutions of \eqref{eqrad}  satisfying \eqref{limiting}, we have $g^2(r)<1$ in $[0,+\infty)$. So, in particular, $g(0)$ must be chosen such that $g(0)^2<1$.
  \end{prop}

To prove the above result, let us first  consider the related conservative system
\begin{equation}\label{eqradcons}
\left\{\begin{aligned}
f'&=g(f^2-a g^2+b)\,,\\
g'&=f(1-g^2)\,.
\end{aligned}\right.
\end{equation}
This system is the Hamiltonian system associated with the energy
\begin{equation}\label{eqenergy}
H(f,g)=\frac{1}{2}f^2(1-g^2)+\frac{a }{4}g^4-\frac{b}{2}g^2.
\end{equation}
As a consequence, to have a complete description of the dynamical system, it is enough to analyze the energy levels of (\ref{eqenergy}). We remind that for every constant $C$,  the $C$-level set is the curve defined by $\Gamma_C=\{(f,g)\;;\; H(f,g)=C\}$. 

\begin{lem}\label{lemenergyprop} For any $a ,b>0$, $H$ has the following properties:
\begin{enumerate}
\item  if $a -b>0$,\begin{enumerate}\item $\nabla H(f,g)=0$ if and only if $(f,g)=(0,0)$, $(f,g)=\left(0,\pm\sqrt{\frac{b}{a }}\right)$, $(f,g)=\left(\pm\sqrt{a -b},1\right)\,$ or if $\,(f,g)=\left(\pm\sqrt{a -b},-1\right)$;
							\item $(f,g)=\left(0,\pm\sqrt{\frac{b}{a }}\right)$ are local minima, and $(f,g)=(0,0)$, $(f,g)=\left(\pm\sqrt{a -b},1\right)$ and $(f,g)=\left(\pm									\sqrt{a -b},-1\right)$ are saddle points of the energy $H$;
			\end{enumerate} 

\item  if $a -b=0$, $\nabla H(f,g)=0$ if and only if $(f,g)=(0,0)$ or $(f,g)=\left(0,\pm1\right)\,$; \\

\item  if $a -b<0$, \begin{enumerate} \item $\nabla H(f,g)=0$ if and only if $(f,g)=(0,0)$ or $(f,g)=\left(0,\pm\sqrt{\frac{b}{a }}\right)$;
							\item $(f,g)=(0,0)$ and $(f,g)=\left(0,\pm\sqrt{\frac{b}{a }}\right)$ are saddle points of the energy $H$.
			\end{enumerate}
\end{enumerate}
\end{lem}
Since we are interested in solutions converging to $(0,0)$ as $r\to+\infty$, we consider the zero level set of $H$. The curve $\Gamma_0=\{(f,g)\;;\; H(f,g)=0\}$ is an algebraic curve of degree $4$ defined by the equation
\begin{equation}\label{eqenergyzero}
a g^4-2(f^2+b)g^2+2f^2=0\,,
\end{equation}
and its behavior depends on the parameters $a$ and $b$. More precisely, the equation (\ref{eqenergyzero}) can be written as 
\begin{equation*}
a \left(g-h_{1,0}(f)\right)\left(g+h_{1,0}(f)\right)\left(g-h_{2,0}(f)\right)\left(g+h_{2,0}(f)\right)=0\,,
\end{equation*}
with $h_{1,0}(f):=\sqrt{\frac{f^2+b+\sqrt{(f^2+b)^2-2a f^2}}{a }}$, $\;h_{2,0}(f):=\sqrt{\frac{f^2+b-\sqrt{(f^2+b)^2-2a f^2}}{a }}$, provided that $h_0(f)=(f^2+b)^2-2a f^2\ge0$. We notice that: 
\begin{enumerate}
\item if $a -2b>0$, then 
\begin{align*}
h_0(f):=&\left(f-\sqrt{a -b+\sqrt{a (a -2b)}}\right)\left(f+\sqrt{a -b+\sqrt{a (a -2b)}}\right)\\
&\left(f-\sqrt{a -b-\sqrt{a (a -2b)}}\right)\left(f+\sqrt{a -b-\sqrt{a (a -2b)}}\right),
\end{align*}
and $h_0(f)\ge 0$ if and only if 
\begin{align*}f\le -\sqrt{a -b+\sqrt{a (a -2b)}}, \, or \ f\ge\sqrt{a -b+\sqrt{a (a -2b)\;}}, \, or\\ 
-\sqrt{a -b-\sqrt{a (a -2b)}}\le f\le \sqrt{a -b-\sqrt{a (a -2b)}}.
\end{align*}

The zero contour line for this case is represented in \textsc{Figure} \ref{figzerocontour1}.
\item if $a -2b=0$, then 
\begin{equation*}
h_0(f)=\left(f-\sqrt{a -b}\right)^2\left(f+\sqrt{a -b}\right)^2\ge 0\,,
\end{equation*}
for all $f$. In this case the equation (\ref{eqenergyzero}) can be written as
$$
(b g^2-f^2)(g^2-1)=0\,,
$$
and the zero contour line is represented in \textsc{Figure} \ref{figzerocontour2}.
\item if $a -2b<0$, then $h_0(f)\ge0$ for all $f$. In this case the zero contour line is defined for all $f$ and it is represented in \textsc{Figure} \ref{figzerocontour3}.
\end{enumerate}
\begin{figure}[h!]
   \begin{minipage}[c]{.47\linewidth}
      \includegraphics{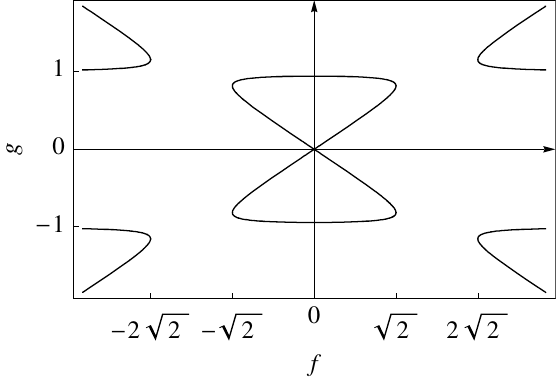}
      \caption{$\Gamma_0$ when $a -2b>0$}
	\label{figzerocontour1}
   \end{minipage}
   \begin{minipage}[c]{.47\linewidth}
      \includegraphics{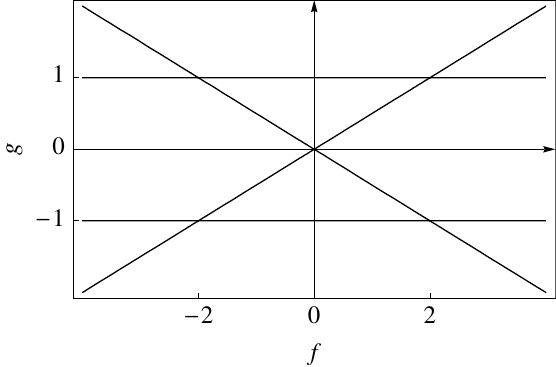}
      \caption{$\Gamma_0$ when $a -2b=0$}
	\label{figzerocontour2}
   \end{minipage}
   \end{figure}
\begin{figure}[h!]
    \begin{minipage}[c]{.47\linewidth}
      \includegraphics{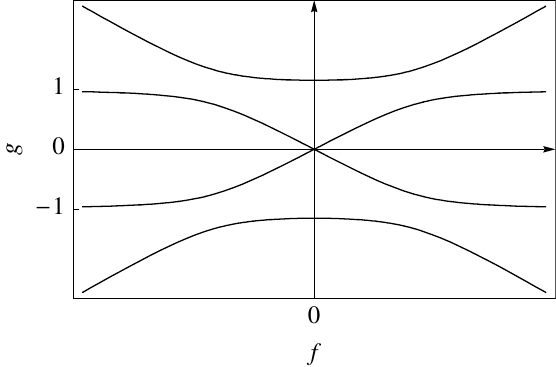}
      \caption{$\Gamma_0$ when $a -2b<0$}
\label{figzerocontour3}
   \end{minipage} \hfill
\end{figure}
More in general, $\Gamma_C$ is the algebraic curve defined by the equation 
\begin{align}\label{eqenergyc}
0&=a g^4-2(f^2+b)g^2+2f^2-4C\nonumber\\
&=a \left(g-h_{1,C}(f)\right)\left(g+h_{1,C}(f)\right)\left(g-h_{2,C}(f)\right)\left(g+h_{2,C}(f)\right)\,,
\end{align}
with 
\begin{align*}
h_{1,C}(f)&:=\sqrt{\frac{f^2+b+\sqrt{(f^2+b)^2-2a (f^2-2C)}}{a }}\,, \\
h_{2,C}(f)&:=\sqrt{\frac{f^2+b-\sqrt{(f^2+b)^2-2a (f^2-2C)}}{a }}\,,
\end{align*}
 and provided that $h_C(f):=(f^2+b)^2-2a (f^2-2C)\ge0$. By a straightforward analysis of the functions $h_{1,C}$, $h_{2,C}$ and using Lemma \ref{lemenergyprop}, we obtain the phase portrait of the dynamical system (\ref{eqradcons}) for different values of the parameters $a $ and $b$. In the pictures \textsc{Figure} \ref{figphaseA2BnegABpos}, \textsc{Figure} \ref{figphaseA2BnegABneg}, \textsc{Figure} \ref{figphaseA2Bzero}, \textsc{Figure} \ref{figphaseABzero} and \textsc{Figure} \ref{figphaseA2Bpos}, the grey part represents the region where the energy is negative. 
\begin{figure}[t!]
\begin{minipage}[c]{.47\linewidth}
      \includegraphics[]{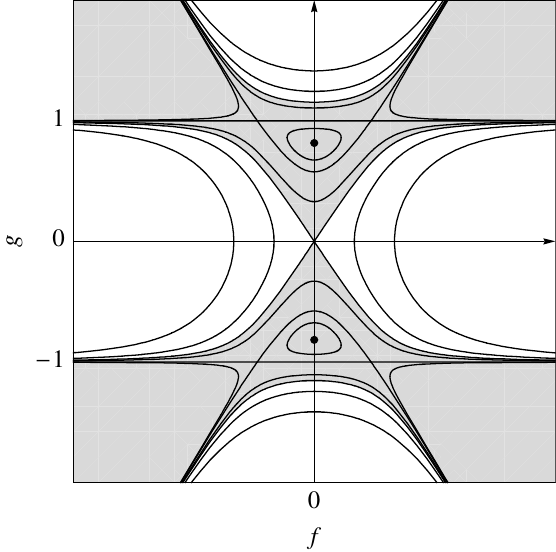}
      \caption{$b<a <2b$}
	\label{figphaseA2BnegABpos}
   \end{minipage}
   \begin{minipage}[c]{.47\linewidth}
      \includegraphics[]{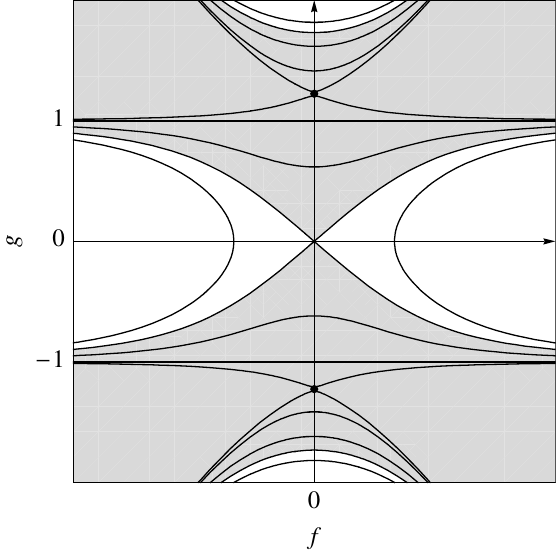}
      \caption{$a <b$}
	\label{figphaseA2BnegABneg}
   \end{minipage}
\begin{minipage}[c]{.47\linewidth}
      \includegraphics[]{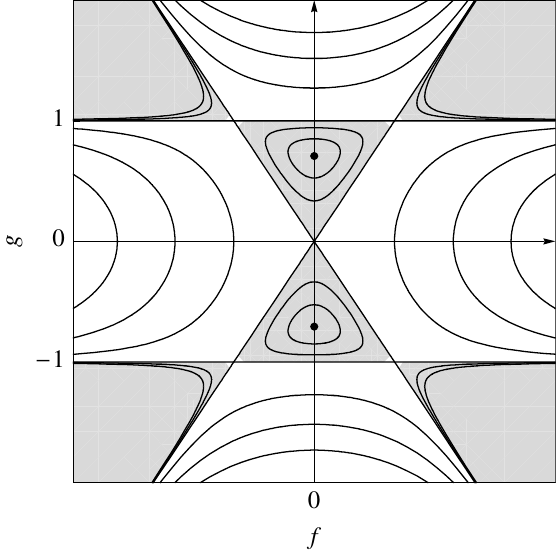}
      \caption{$a -2b=0$}
	\label{figphaseA2Bzero}
   \end{minipage}
   \begin{minipage}[c]{.47\linewidth}
      \includegraphics[]{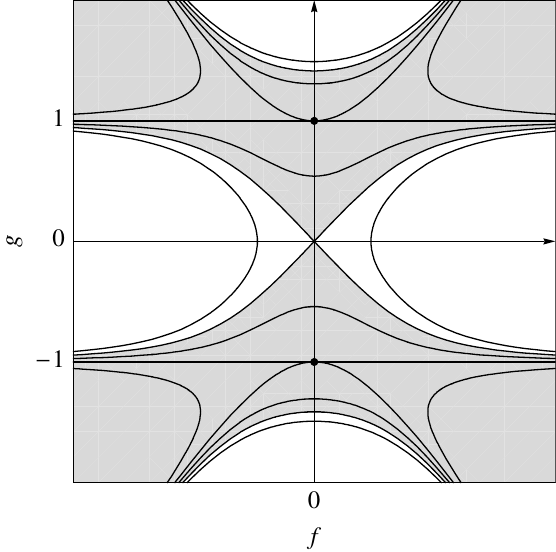}
      \caption{$a -b=0$}
	\label{figphaseABzero}
   \end{minipage}
\end{figure}
\begin{figure}[th!]
\begin{center}
\includegraphics[scale=0.6]{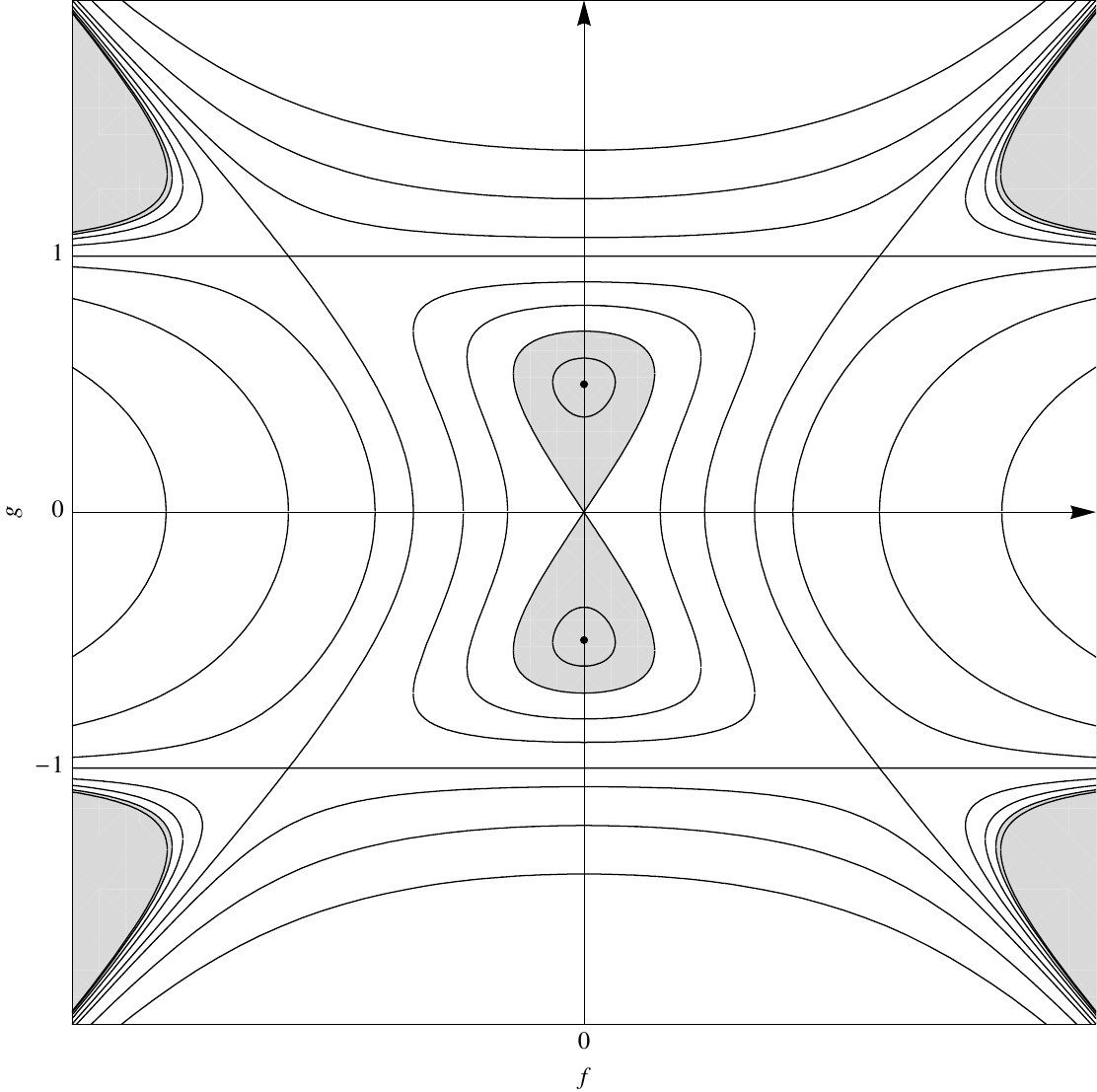}
      \caption{$a -2b>0$}
	\label{figphaseA2Bpos}
\end{center}
\end{figure}
\newpage
Now, we are ready to study the non conservative system (\ref{eqrad}). We begin with the following lemma.
\begin{lem}\label{lemexistencesol}
Let $x\in \mathbb{R}$. For any $a ,b>0$, there is $\tau>0$ and $(f_x,g_x)\in\mathcal{C}^1\left([0,\tau],\mathbb{R}^2\right)$ unique solution of (\ref{eqrad}) satisfying $f_x(0)=0$, $g_x(0)=x$. Moreover, $(f_x,g_x)$ can be extended on a maximal interval $[0,R_x)$ with either $R_x=+\infty$ or $R_x<+\infty$ and $\lim_{r\to { R_x}}|f_x|+|g_x|=+\infty$. Furthermore, $(f_x,g_x)$ depends continuously on $x$, uniformly on $[0,R]$ for any $R<R_x$.
\end{lem}
\begin{proof}
As in \cite{cazenavevazquez}, it is enough to write
\begin{align*}
f(r)&=\frac{1}{r^2}\int_{0}^rs^2g(s)(f^2(s)-a g^2(s)+b)\,ds\,,\\
g(r)&=x+\int_{0}^rf(s)(1-g^2(s))\,ds\,,
\end{align*}
and note that the right hand side of (\ref{eqrad}) is a Lipschitz continuous function of $(f,g)$. The lemma follows from a classical contraction mapping argument.
\end{proof}
\begin{rem}\label{remsolzero} For any $a ,b>0$, $(f(r),g(r))\equiv(0,0)$ is the unique continuous solution of (\ref{eqrad}) satisfying $f(0)=0$, $g(0)=0$ and defined on $[0,+\infty)$. 
\end{rem}
\begin{rem}\label{remsolone} For $a-b> 0$, \begin{align}f(r)=&\left\{\begin{aligned}&\frac{1}{r}-\sqrt{a-b}\,\mathrm{Coth}\left(\sqrt{a-b}\,r\right)&r>0\\&0&r=0\end{aligned}\right.\\g(r)=&\,1\end{align} is the unique continuous solution of (\ref{eqrad}) satisfying $f(0)=0$, $g(0)=1$ and defined on $[0,+\infty)$. Moreover, $(-f(r),-g(r))$ is the unique continuous solution of (\ref{eqrad}) satisfying $f(0)=0$, $g(0)=-1$ and defined on $[0,+\infty)$.
\end{rem}

\begin{lem}\label{lemvarenergy}Let $x\in \mathbb{R}$ and $a ,b>0$. Then for any $r\in [0,R_x)$ we have
\begin{equation}\label{eqvarenergy}
\frac{d}{dr}H(f_x(r),g_x(r))=-\frac{2}{r}f_x^2(r)(1-g_x^2(r)).
\end{equation}
\end{lem}
\begin{proof}
The equation (\ref{eqvarenergy}) is derived by a straightforward calculation
\begin{align*}
\frac{d}{dr}H(f_x(r),g_x(r))=&f_x(r)(1-g_x^2(r))f_x'(r)+(-f_x^2(r)+a g_x^2(r)-b)g_x(r)g_x'(r)\\
=&-\frac{2}{r}f_x^2(r)(1-g_x^2(r)).
\end{align*}
\end{proof}
\begin{lem}\label{lemA2Bneg}Let  $a ,b>0$ such that $a -2b<0$ and let $(f_x,g_x)$ be the solution of (\ref{eqrad}) satisfying $f_x(0)=0$, $g_x(0)=x$. If $x\neq0$ there is no solution that satisfies $\lim\limits_{r\to+\infty}(f_x(r),g_x(r))=(0,0)$.
\end{lem}
\begin{proof}
Let $(f_x(r),g_x(r))$ be a solution of (\ref{eqrad}) satisfying $f_x(0)=0$, $g_x(0)=x$ and defined on $[0,+\infty)$. If $\lim\limits_{r\to+\infty}(f_x(r),g_x(r))=(0,0)$, then $\forall\eps>0$, $\exists\delta_\eps>0$ such that $|g_x(r)|<\eps$ for $r>\delta_\eps$. In particular, $|g_x(r)|<1$ and the energy $H(f_x(r),g_x(r))$ is not increasing on $(\bar\delta,+\infty)$ for some $\bar\delta>0$ (see Lemma \ref{lemvarenergy}). This implies that, $\forall r>\bar\delta$,
$$
H(f_x(r),g_x(r))\ge\lim\limits_{r\to+\infty}H(f_x(r),g_x(r))=0.
$$
We define by $\Omega$ the set 
\begin{align}\label{eqOmega}
\Omega&=\left\{(f_x(r),g_x(r))\in \mathbb{R}^2\;;\;H(f_x(r),g_x(r))\ge0\ \mbox{and}\ |g_x(r)|<1\right\}\nonumber\\
&=\left\{(f_x(r),g_x(r))\in \mathbb{R}^2\;;\; |g_x(r)|\le h_{2,0}(f_x(r))\right\}\,,
\end{align}
with
$$
h_{2,0}(f_x(r))=\sqrt{\frac{f_x^2(r)+b-\sqrt{(f^2_x(r)+b)^2-2a f^2_x(r)}}{a }}.
$$
Then, $(f_x(r),g_x(r))\in\Omega$ for all $r>\bar\delta$. Moreover, we remark that, if $a -2b<0$, then $h_{2,0}(f_x(r))<1$ for all $f_x(r)\in \mathbb{R}$.
 
Suppose now that $x^2<\frac{2b}{a }$. Since $H(0,x)<0$ and thanks to the continuity of $H(f_x(r),g_x(r))$, there exists $\bar r \le \bar\delta$ such that $H(f_x(\bar r),g_x(\bar r))=0$; more precisely let $$\bar r =\inf\{r\in(0,\bar\delta]|(f_x(r),g_x(r))\in\partial\Omega\}.$$ Since $|g_x(\bar r)|= h_{2,0}(f_x(\bar r))<1$, the energy $H(f_x(r),g_x(r))$ is not increasing on a neighborhood of $\bar r$. As, for all $\eps>0$, $H(f_x(\bar r-\eps),g_x(\bar r-\eps))<0=H(f_x(\bar r),g_x(\bar r))$, we obtain a contradiction.

Now, suppose that $x\ge\sqrt{\frac{2b}{a }}$; in this case $H(0,x)\ge0$ and $g_x(0)>1$. Since $\lim\limits_{r\to+\infty}(f_x(r),g_x(r))=(0,0)$ and $g_x$ is a continuous function, there exists $\bar r \le \bar\delta$ such that $g_x(\bar r)=1$; more precisely let $\bar{r}=\inf\{r\in(0,+\infty)|g_x(r)=1\}$. This implies that $g_x(r)>1$ for all $r\in [0,\bar r)$, and, as a consequence, $H(f_x(r),g_x(r))$ is not decreasing on $[0,\bar r)$. Then 
$$
H(0,x)\le H(f_x(\bar r),1)=\frac{a -2b}{4}<0.
$$
Hence, we obtain a contradiction and $g_x(r)>1$ for all $r\in[0,+\infty)$.

Finally, with the same arguments, we prove that if $x\le-\sqrt{\frac{2b}{a }}$, then $g_x(r)<-1$ for all $r\in[0,+\infty)$.
\end{proof}

\begin{lem}\label{lemABpos}Let  $a ,b>0$ such that $a -b>0$ and let $(f_x,g_x)$ be the solution of (\ref{eqrad}) satisfying $f_x(0)=0$, $g_x(0)=x$. If $x^2>1$, then $g_x^2(r)>1$ for all $r\in [0,R_x)$.
\end{lem}
\begin{proof} To prove this lemma, we use the monotonicity properties of the function $F(x)=\frac{a }{4}x^4-\frac{b}{2}x^2$. In particular, we use the fact that $F(x)$ is decreasing on $\left(-\infty, -\sqrt{\frac{b}{a }}\right)$ and increasing on $\left(\sqrt{\frac{b}{a }},+\infty\right)$. 

Let $g_x(0)=x>1$ and suppose, by contradiction, that there exists $r_0$ such that $g_x(r_0)=1$ and $g_x(r)>1$ for all $r\in [0,r_0)$.
As a consequence of Lemma \ref{lemvarenergy}, the energy $H(f_x(r),g_x(r))$ is not decreasing on $[0,r_0)$, which means
$$
H(0,x)\le H(f_x(r_0),g_x(r_0)) \; \mbox{, or equivalently, } \;
F(x) \le F(1).
$$
Since $\sqrt{\frac{b}{a }}<1$, the above inequality  contradicts the monotonicity of $F$. As a conclusion, $g_x(r)>1$ for all $r\in [0,R_x)$.
In the same way, using the fact that energy is not decreasing if $g_x(r)<-1$, we prove that if $g_x(0)=x<-1$, then $g_x(r)<-1$ for all $r\in [0,R_x)$.
\end{proof}

\begin{lem}\label{lemA2Bzero}Let  $a ,b>0$ such that $a =2b$ and let $(f_x,g_x)$ be the solution of (\ref{eqrad}) satisfying $f_x(0)=0$, $g_x(0)=x$. If $x\neq0$ there is no solution that satisfies $\lim\limits_{r\to+\infty}(f_x(r),g_x(r))=(0,0)$.
\end{lem}
\begin{proof}
Let $(f_x(r),g_x(r))$ be a solution of (\ref{eqrad}) satisfying $f_x(0)=0$, $g_x(0)=x$ and defined on $[0,+\infty)$.
Suppose first that $x^2>1$. Then, thanks to Lemma \ref{lemABpos}, $g_x^2(r)>1$ for all $r\in [0,+\infty)$. 
Next, suppose that $x^2<1$. As before, if $\lim\limits_{r\to+\infty}(f_x(r),g_x(r))=(0,0)$, then $\forall\eps>0$, $\exists\delta_\eps>0$ such that $|g_x(r)|<\eps$ for $r>\delta_\eps$. In particular, $|g_x(r)|<1$ and the energy $H(f_x(r),g_x(r))$ is not increasing on $(\bar\delta,+\infty)$, for some $\bar\delta>0$. This implies that, $\forall r>\bar\delta$,
$$
H(f_x(r),g_x(r))\ge\lim\limits_{r\to+\infty}H(f_x(r),g_x(r))=0.
$$
Since $H(0,x)<0$ and thanks to the continuity of $H(f_x(r),g_x(r))$, there exists $\bar r \le \bar\delta$ such that $H(f_x(\bar r),g_x(\bar r))=0$; more precisely let $$\bar r =\inf\{r\in(0,\bar\delta]|H(f_x(r),g_x(r))=0\}.$$ By definition of $\bar r$, we have that $g^2(r)<1$ for all $r\in[0,\bar r) $; this implies that the energy is not increasing on $[0,\bar r )$. Then 
$$
0= H(f_x(\bar r),g_x(\bar r))\le H(0,x),
$$ 
 a contradiction.
Finally, thanks to Remark \ref{remsolone}, if $x^2=1$, then $g^2(r)=1$ for all $r\in [0,+\infty)$.
\end{proof}


\begin{lem}\label{lemA2Bpos}Let  $a ,b>0$ such that $a -2b>0$ and let $(f_x,g_x)$ be the solution of (\ref{eqrad}) satisfying $f_x(0)=0$, $g_x(0)=x$. If $x^2<1$, then $g_x^2(r)<1$ for all $r\in [0,R_x)$.\\
\end{lem}
\begin{proof} To prove this lemma, we use the monotonicity of the function $F(x)=\frac{a }{4}x^4-\frac{b}{2}x^2$ and the fact that $F(x)\le0$ in $ \left[-\sqrt{\frac{2b}{a }}, \sqrt{\frac{2b}{a }}\right]$ . In particular, since $ \sqrt{\frac{2b}{a}}<1$, $F(x)<F(1)$ for all $x$ such that $x^2<1$.

Let $g_x(0)=x$ such that $x^2<1$ and suppose, by contradiction, that there exists $r_0$ such that $g_x^2(r_0)=1$ and $g_x^2(r)<1$ for all $r\in [0,r_0)$.
As a consequence of Lemma \ref{lemvarenergy}, the energy $H(f_x(r),g_x(r))$ is not increasing on $[0,r_0)$, that means
$$
H(0,x)\ge H(f_x(r_0),g_x(r_0)) \; \mbox{, or equivalently, } \;
F(x) \ge F(1).
$$
The above inequality contradicts the properties of $F$. As a conclusion, $g_x^2(r)<1$ for all $r\in [0,R_x)$.
\end{proof}

\begin{prop}\label{propconditions} Let  $a ,b>0$ such that $a -2b>0$ and let $(f_x,g_x)$ be the solution of (\ref{eqrad}) satisfying $f_x(0)=0$, $g_x(0)=x$ and such that $\lim\limits_{r\to+\infty}(f_x(r),g_x(r))=(0,0)$. Then $x^2<1$ and $g_x^2(r)<1$ for all $r\in [0,+\infty)$.
\end{prop}
\begin{proof}
Let $(f_x,g_x)$ be the solution of (\ref{eqrad}) satisfying $f_x(0)=0$, $g_x(0)=x$ and such that $\lim\limits_{r\to+\infty}(f_x(r),g_x(r))=(0,0)$; then $\forall\eps>0$, $\exists\delta_\eps>0$ such that $|g_x(r)|<\eps$ for $r>\delta_\eps$. Hence, as a consequence of Lemma \ref{lemABpos} and Remark \ref{remsolone}, we have $x^2<1$.

Finally, thanks to Lemma \ref{lemA2Bpos}, we can conclude that  $g_x^2(r)<1$ for all $r\in [0,+\infty)$.
\end{proof}

\begin{proof}[Proof of Proposition \ref{conditionsparameters}] It follows readily from the above series of lemmata.
\end{proof}

\section{Existence of solutions}\label{existence}

In this section, we prove Theorem \ref{thmexistence} that is, we show that whenever $a>2b$, there exists at least one  nontrivial solution of the system
\begin{equation}\label{eqradbis}
\left\{\begin{aligned}
f'+\frac{2}{r}f&=g(f^2-a g^2+b)\,,\\
g'&=f(1-g^2)\,,
\end{aligned}\right.
\end{equation}
which goes  to $(0,0)$ as $r\to+\infty$ and such that $f(0)=0$. Moreover $f\le 0$ and $0\le g\le 1$ (or equivalently, $f\ge 0$ and $-1\le g\le 0$).

In what follows, let then $a,b>0$ be such that 
$\,a -2b>0\,$
and let $(f_x,g_x)$ be the solution of (\ref{eqradbis}) satisfying $f_x(0)=0$, $g_x(0)=x$.

We remind that (\ref{eqradbis}) is associated to the conservative system
\begin{equation}\label{eqradconsbis}
\left\{\begin{aligned}
f'&=g(f^2-a g^2+b)\,,\\
g'&=f(1-g^2)\,,
\end{aligned}\right.
\end{equation}
and we define  $\mathcal{A}=\{(f_0,g_0)\in\mathbb{R}^2|\, 2f_0^2-ag_0^2-(a-2b)\le 0,\; g_0^2\le 1\}$, the  set of admissible initial conditions. We recall that the energy of the system is given by
\begin{equation*}
H(f,g)=\frac{1}{2}f^2(1-g^2)+\frac{a }{4}g^4-\frac{b}{2}g^2.
\end{equation*}
\begin{rem*} If $g_0^2\le 1$, $(f_0,g_0)\in \mathcal{A}\,$ if and only if $\,H(f_0,g_0)\le H(0,1)=\frac{a-2b}{4}$. \end{rem*}

\begin{lem}\label{lemboundedorbit} Let $(f_0,g_0)\in \mathcal{A}$ and let $(f,g)$ be the solution of the conservative system (\ref{eqradconsbis}) with initial data $(f_0,g_0)$. Then $(f(r),g(r))\in \mathcal{A}$ for all $r\in[0,+\infty)$.  In particular, $(f(r),g(r))$ is bounded for all $r\in[0,+\infty)$.
\end{lem}

\begin{proof}
First of all, let $(f_0,g_0)\in \mathcal{A}$ such that $g_0^2<1$. Since in the case of autonomous systems two different trajectories cannot intersect, if $g_0^2<1$ then, $g^2(r)<1$ for all $r\in[0,+\infty)$. Moreover, $H(f,g)=H(f_0,g_0)\le\frac{a-2b}{4}$. This implies 
\begin{align*}
&2f^2(r)-ag^2(r)-(a-2b)\le 0\,,
\end{align*}
for all $r\in (0,+\infty)$.

Now suppose that $g_0=1$ ($g_0=-1$, respectively). In this case, $g(r)=1$ ($g(r)=-1$, respectively) for all $r\in (0,+\infty)$ and $f(r)$ solves 
\begin{equation}
\label{eqradconsfg1}
\left\{\begin{aligned}
f'(r)&=f^2(r)-a+b\,,\\ 
f(0)&=f_0\,.
\end{aligned}\right..
\end{equation}
Hence, if $f_0=\pm \sqrt{a-b}$, then $f(r)=\pm  \sqrt{a-b}$ for all $r\in[0,+\infty)$. Finally, if $f_0\neq\pm \sqrt{a-b}$, we solve the Cauchy problem (\ref{eqradconsfg1}) and we obtain
$$
f(r)=-2\sqrt{a-b}\left(e^{2\sqrt{a-b}\,r}\left(1-\frac{2\sqrt{a-b}}{f_0+\sqrt{a-b}}\right)-1\right)^{-1}-\sqrt{a-b}\,.
$$
Note that $f(r)$ is bounded for all $r\in(0,+\infty)$. Indeed, if $(f_0,1)\in\mathcal{A}$ and $f_0\neq\pm \sqrt{a-b}$, then $-\sqrt{a-b}<f_0<\sqrt{a-b}$ and 
$f'(r)\le0$ for all $r\in(0,+\infty)$. Hence, $f_0>f(r)>-\sqrt{a-b}$ for all $r\in(0,+\infty)$.
\end{proof}

\begin{lem}\label{lemconvconssol} 
Let $(f_0,g_0)\in \mathcal{A}$. Let $(f,g)$ be the solution of (\ref{eqradconsbis}) with initial data $(f_0,g_0)$. Let $(f^0_n,g^0_n)\in\mathcal{A}$ and $\rho_n$ be such that 
\begin{align*}
\lim\limits_{n\to+\infty}\rho_n=+\infty&\mbox{ and }\lim\limits_{n\to+\infty}(f^0_n,g^0_n)=(f_0,g_0).
\end{align*}
Let $(f_n,g_n)$ be a solution of 
\begin{equation*}
\left\{\begin{aligned}
f'_n+\frac{2}{\rho_n+r}f_n&=g_n(f_n^2-a g_n^2+b)\,,\\
g_n'&=f_n(1-g_n^2)\,,
\end{aligned}\right.
\end{equation*}
such that $f_n(0)=f_n^0$, $g_n(0)=g_n^0$, and let $[0,\tau_n)$ be the maximal existence interval of $(f_n,g_n)$. Then
\begin{enumerate}
\item $\lim\limits_{n\to+\infty}\tau_n=+\infty$,
\item $(f_n,g_n)$ converges to $(f,g)$ uniformly on bounded intervals.
\end{enumerate} 
\end{lem}

\begin{proof} For the proof of this lemma, we proceed exactly as in the proof of Lemma $2.5$ of \cite{cazenavevazquez}; the only difference is that we do not work in all $\mathbb{R}^2$ but in $\mathcal{A}$. For the reader's convenience, we rewrite the proof in our case. 

We have to prove that for any $R<+\infty$, $\lim\limits_{n\to+\infty}\tau_n>R$ and $(f_n,g_n)$ converges to $(f,g)$ uniformly on $[0,R]$. Then consider $R<+\infty$. Thanks to Lemma \ref{lemboundedorbit}, if $(f_0,g_0)\in\mathcal{A}$, then $\bigcup\limits_{r\in[0,R]}(f(r),g(r))$ is a bounded set. Hence,  applying a contraction mapping argument, we obtain that there exists $\delta>0$ and $\tau>0$ satisfying the following property: for any $\rho\ge 1$, $r_0\in[0,R]$ and $(w_0,z_0)\in \mathcal{A}$ such that $|w_0-f(r_0)|+|z_0-g(r_0)|\le\delta$, there exists a solution $(w,z)\in\mathcal{C}^1([0,\tau],\mathbb{R}^2)$ of 
\begin{equation}\label{eqconvcons}
\left\{\begin{aligned}
w'+\frac{2}{\rho+r}w&=z(w^2-a z^2+b)\,,\\
z'&=w(1-z^2)\,,
\end{aligned}\right.
\end{equation}
with $w(0)=w_0$, $z(0)=z_0$\; and $\sup\limits_{r\in[0,\tau]}|w(r)|+|z(r)|\le 2M$, where $M=\sup\limits_{r\in[0,\tau]}|f(r)|+|g(r)|$.

Next, let $K$ the Lipschitz constant of the second term of the system (\ref{eqconvcons}) on the ball of radius $2M$. We get, for all $r\in[0,\tau]$, 
\begin{align*}
|w(r)-&\,f(r+r_0)|+|z(r)-g(r+r_0)|\le |w_0-f(r_0)|+|z_0-g(r_0)|\\
& +\int_0^r\frac{2}{\rho+s}|w(s)|\,ds+\int_0^rK(|w(s)-f(s+r_0)|+|z(s)-g(s+r_0)|)\,ds\\
\le&  |w_0-f(r_0)|+|z_0-g(r_0)|+\frac{4M\tau}{\rho}\\
&+\int_0^rK(|w(s)-f(s+r_0)|+|z(s)-g(s+r_0)|)\,ds.
\end{align*} 
Thus, thanks to Gronwall's lemma, we have
$$
|w(r)-f(r+r_0)|+|z(r)-g(r+r_0)|\le\left(|w_0-f(r_0)|+|z_0-g(r_0)|+\frac{4M\tau}{\rho}\right)e^{K\tau}\,,
$$
for all $r\in[0,\tau]$.
Finally, we apply this argument with $\rho=\rho_n$, $r_0=0$, $w_0=f_n^0$, $z_0=g_n^0$. We obtain $\lim\limits_{n\to+\infty}\tau_n\ge\tau$ and $(f_n,g_n)$ converges uniformly to $(f,g)$ on $(0,\tau)$.
To conclude, we iterate the above procedure  $N$ times with $N\tau\le R\le (N+1)\tau$.
\end{proof}

\begin{lem}\label{lempropertiessol} Let  $a ,b>0$ such that $a -2b>0$ and let $(f_x,g_x)$ be the solution of (\ref{eqradbis}) satisfying $f_x(0)=0$, $g_x(0)=x$. If $x^2<1$, then $g_x^2(r)<1$ and $f_x^2(r)<a-b$ for all $r \in [0,R_x)$ and $R_x=+\infty$. Moreover, $(f_x(r),g_x(r))\in\mathcal{A}$, for all $r\in[0,+\infty)$.
\end{lem}

\begin{proof} 
By Lemma \ref{lemA2Bpos} we have $g_x^2(r)<1$ on $[0,R_x)$. Then, applying Lemma \ref{lemvarenergy}, we obtain that the energy is non-increasing. Thus, 
$$
H(f_x(r),g_x(r))\le H(0,x)<\frac{a-2b}{4}\,,\;\forall r\in[0,R_x)\,,$$
and by the remark following the definition of the set $\mathcal{A}$, $(f_x(r),g_x(r))\in\mathcal{A}$ and $f_x^2(r)<a-b$, for all $r\in[0,R_x)$. In particular, $R_x=+\infty$.
\end{proof}

\begin{lem}\label{lemexpdecay} Let $a-2b>0$, $0<x<1$ and $f_x(r)<0$, $g_x(r)>0$ for $r\in (0,+\infty)$. Then there exists a positive constant $C$ such that 
\begin{align}\label{eqexpdecay}
0\le -f_x(r), g_x(r)&\le C\exp(-K_{a,b}\,r)\,,
\end{align}
with $K_{a,b}=\min\left\{\frac{b}{2},\frac{2a-b}{2a}\right\}$.
\end{lem}

{\begin{rem*}The ideas behind the proof of this lemma are inspired by the work of Cazenave and Vazquez \cite{cazenavevazquez}. 
\end{rem*}
}

\begin{proof} Applying Lemma \ref{lempropertiessol} we have $g_x^2(r)<1$. Thus $g_x$ is decreasing on $(0,+\infty)$ and there exists $\delta\ge 0$ such that $\lim\limits_{r\to+\infty}g_x(r)=\delta<1$. 

Suppose for the moment $\delta=0$. Then $\lim\limits_{r\to+\infty}g_x(r)=0$ and $g_x^2(r)\le\frac{b}{2a}$ for r large enough.
Considering (\ref{eqradbis}) we have
$$
f_x'\ge-\frac{2}{r}f_x+\frac{b}{2}g_x\;,\quad
g_x'\le\frac{2a-b}{2a}f_x\,.
$$
Hence, for $r$ large enough,
$$
(g_x-f_x)'+K_{a,b}(g_x-f_x)\le 0\,,
$$
with $K_{a,b}=\min\left\{\frac{b}{2},\frac{2a-b}{2a}\right\}$. Integrating this inequality, we obtain 
$$
-f_x(r)+g_x(r)\le C\exp(-K_{a,b}r).
$$
Therefore it only remains to prove that $\delta=0$. By contradiction, suppose $\delta>0$.

Let $\{r_n\}_n$ be a sequence such that $\lim\limits_{n\to+\infty}r_n=+\infty$ and $\lim\limits_{n\to+\infty}f_x(r_n)=k$. Let $(u,v)$ be the solution of (\ref{eqradconsbis}) with initial data $(k,\delta)$. It follows from Lemma \ref{lemconvconssol} that $(f_x(r_n+\cdot),g_x(r_n+\cdot))$ converges uniformly to $(u,v)$ on bounded intervals. Since $\lim\limits_{n\to+\infty}g_x(r_n+r)=\delta$ for any $r>0$, we have $v(r)=\delta<1$ for any $r\ge0$. Hence, from the second equation of (\ref{eqradconsbis}) and applying Lemma \ref{lemenergyprop}, we have $k=0$ and $\delta=\sqrt{\frac{b}{a}}$. As a conclusion, this means
$$
\lim\limits_{r\to+\infty}(f_x(r),g_x(r))=\left(0,\sqrt{\frac{b}{a}}\right).
$$
Now, we put $g_x=\sqrt{\frac{b}{a}}+w$. First of all, we remark that 
$$
\lim\limits_{r\to+\infty}(f_x(r),w(r))=\left(0,0\right)\,,
$$
and $w\ge 0$ on $[0,+\infty)$.
Furthermore, the equations for $(f_x,w)$ are given by 
$$
\left\{\begin{aligned}
f_x'+\frac{2}{r}f_x&=\left(\sqrt{\frac{b}{a}}+w\right)\left(f^2-a \left(\sqrt{\frac{b}{a}}+w\right)^2+b\right)\,,\\
w'&=f_x\left(1-\left(\sqrt{\frac{b}{a}}+w\right)^2\right)\,,
\end{aligned}\right.
$$
which implies that as $\,r\to +\infty$, 
\begin{equation}\label{eqproofexpdecay}
\left\{\begin{aligned}
f_x'+\frac{2}{r}f_x&=-2bw+o(|w|+|f_x|)\,,\\
w'&=f_x\left(1-\frac{b}{a}\right)+o(|w|+|f_x|)\,,
\end{aligned}\right. . 
\end{equation}
Then,
$$
(w+f_x)'=\frac{1}{2}\left(\frac{a-b}{a}\right)(w+f_x)+\left(\frac{a-b}{2a}-\frac{2}{r}\right)f_x-\left(\frac{a-b}{2a}+2b\right)w+o(|w|+|f_x|).
$$
As a consequence, for $r$ large enough, we obtain
$$
(w+f_x)'-\frac{1}{2}\left(\frac{a-b}{a}\right)(w+f_x)\le0\,.
$$
Therefore, $\exp\left(-\frac{a-b}{2a}r\right)(w+f_x)$ is non-increasing for $r$ large, and since $$\lim\limits_{r\to+\infty}\exp\left(-\frac{a-b}{2a}r\right)(w+f_x)=0,$$ we should have $w\ge-f_x>0$ for $r$ large. This implies that $o(|w|+|f_x|)$ can be replaced by $o(|w|)$.
To conclude, we put $\tilde f_x=-f_x$ and we remark that $\tilde f_x>0$ on $(0,+\infty)$ and $w\ge\tilde f_x$ for $r$ large. From the first equation of (\ref{eqproofexpdecay}) we obtain, for $r$ large,
$
\tilde f_x'\ge \left(\frac{3b}2-\frac{2}{r}\right) \tilde f_x\ge b\tilde f_x$, which implies 
$\tilde f_x(r) \ge C e^{br}
$ for some $C>0$, 
contradicting $\lim\limits_{r\to+\infty}\tilde f_x(r)=0$. Thus the lemma is proved.
\end{proof}

We are now ready to prove the main result of this paper. The proof of this theorem is given in several steps. Some parts of this proof are inspired by the work of Cazenave and Vazquez \cite{cazenavevazquez}. For our problem, we have to introduce new ideas to deal with the fact that the initial condition for $g$ have to be bounded.

We define the set $I$ as follows:
\begin{equation}\label{eqsetI}
I=\left\{x>\sqrt{\frac{b}{a}} \left|\right. \exists\, r_x\in (0,R_x), f_x<0 \mbox{ and } g_x>0 \mbox{ on } (0, r_x), f_x(r_x)=0 \right\}.
\end{equation}  
In the first step, we show that $I$ is a non empty open set. In the second step, we prove that $I$ is a bounded set and that $\sup I<1$. Finally, in the third step, we conclude taking $\tilde x=\sup I$ as initial condition.

Since we take $\tilde x=\sup I$ as initial condition for $g$, it follows from Proposition \ref{conditionsparameters} that $\sup I$ must be strictly smaller than $1$. Actually this is the main difficulty to deal with and represent the principal difference between our proof and Cazenave-Vazquez's one.\\

\noindent\textit{Step $1$.} $I$ is a non empty open set.

To prove that $I$ is a non empty set, we show that  $\left(\sqrt{\frac{b}{a}},\sqrt{\frac{2b}{a}}\right)\subset I$. Indeed, let $x\in \left(\sqrt{\frac{b}{a}},\sqrt{\frac{2b}{a}}\right)\subset \left(\sqrt{\frac{b}{a}},1\right)$. From Lemma \ref{lempropertiessol}, we have $g_x^2(r)<1$, $f^2_x(r)<a-b$ and $R_x=+\infty$. Hence the energy is non-increasing and $$H(f_x(r),g_x(r))\le H(0,x) <H\left(0,\sqrt{\frac{2b}{a}}\right)=0\,,$$ for $r>0$. 
On the other hand, $H(y,0)=\frac{1}{2}y^2 \ge 0$ for $y\in \mathbb R$. As a consequence, $g_x$ cannot vanish in a finite time. Therefore $g_x(r)>0$ for all $r\in (0,+\infty)$.

Next, using the first equation of (\ref{eqradbis}), we obtain $f'_x(0)=x(b-ax^2)/3<0$. Hence, $f_x(r)<0$ for $r>0$  small enough. Thus it remains to prove that $f_x(r)$ vanishes for some $r>0$. Suppose by contradiction that $f_x(r)<0$ for all $r\in(0,+\infty)$. Then, we can apply Lemma \ref{lemexpdecay} which asserts that $\lim\limits_{r\to+\infty}(f_x(r),g_x(r))=(0,0)$. This implies $\lim\limits_{r\to+\infty}H(f_x(r),g_x(r))=0$, which is a contradiction since $H(f_x(r),g_x(r))\le H(0,x) <0$. Therefore $f_x$ has to vanish and $x\in I$.

The open character of $I$ is a consequence of the continuous dependence of $(f_x, g_x)$ on $x$.\\

\noindent\textit{Step $2$.} $\sup I <1$.\\

First of all, we show that $\sup I\le 1$ and we state some properties of solutions of (\ref{eqradbis}) with initial condition $x\in I$.

\begin{lem}\label{lempropertiesI1} Let $x\in I$. Then  $0\le g_x(r_x)\le \sqrt{\frac{b}{a}}$ and $H(f_x(r_x),g_x(r_x))\le 0$.
\end{lem}

\begin{proof}
Suppose by contradiction $g_x(r_x)> \sqrt{\frac{b}{a}}$. Using the first equation of (\ref{eqradbis}), we get $$f_x'(r_x)=g_x(r_x)(b-ag_x^2(r_x))<0.$$ Hence, $f_x$ is decreasing in a neighborhood of $r_x$; this implies that for $\eps>0$ sufficiently small $f_x(r_x-\eps)>f_x(r_x)=0$, which contradicts the definition of $r_x$. 
Finally, if $0\le g_x(r_x)\le \sqrt{\frac{b}{a}}$, then $H(f_x(r_x),g_x(r_x))=\frac{a}{4}g^4_x(r_x)-\frac{b}{2}g_x^2(r_x)\le 0$.
\end{proof}
As a consequence of Lemma \ref{lempropertiesI1}, $I\subset  \Big(\sqrt{\frac{b}{a}},1\Big)$. Indeed, if $x\ge 1$ then, thanks to Lemma \ref{lemABpos} and Remark \ref{remsolone},  $g_x(r)\ge 1$ for all $r\in [0,R_x)$. Hence $\sup I\le 1$. 

\begin{lem}\label{lempropertiesI2} Let $x\in I$. Then  $0<-f_x(r)\le \sqrt{\frac{a}{2}}\,g_x(r)$, for all  $r\in(0,r_x)$.
\end{lem}

\begin{proof}
Let $h_x(r)=\sqrt{\frac{a}{2}}\,g_x(r)+f_x(r)$ for $0\le r\le r_x$. We have $h_x(0)=\sqrt{\frac{a}{2}}\,x>0$ and $h_x(r_x)=\sqrt{\frac{a}{2}}\,g_x(r_x)>0$. Suppose by contradiction that there exists  $\rho\in(0,r_x)$ such that $h_x(\rho)<0$. Then, there exist $\rho_1,\rho_2\in(0,r_x)$ such that 
\begin{align*}
\rho_1<\rho, \; h_x(\rho_1)=0 \mbox{ and } h'_x(\rho_1)\le 0\,, \\
\rho_2>\rho, \; h_x(\rho_2)=0 \mbox{ and } h'_x(\rho_2)\ge 0.
\end{align*} 
Next, we remark that 
\begin{align*}
\sqrt{\frac{a}{2}}\,h'_x(r)=&\left(\frac{a-2b}{2}+\frac{a}{2}g^2_x(r)-f^2_x(r)-\frac{2}{r}\sqrt{\frac{a}{2}}\right)f_x(r)\\
&+(f^2_x(r)-ag^2_x(r)+b)h_x(r). 
\end{align*}
Hence, 
\begin{align*}
0\ge \sqrt{\frac{a}{2}}\,h'_x(\rho_1)&=\left(\frac{a-2b}{2}+\frac{a}{2}g^2_x(\rho_1)-f^2_x(\rho_1)-\frac{2}{\rho_1}\sqrt{\frac{a}{2}}\right)f_x(\rho_1)\\
&=\left(\frac{a-2b}{2}-\frac{2}{\rho_1}\sqrt{\frac{a}{2}}\right)f_x(\rho_1)
\end{align*}
and, since $f_x(\rho_1)<0$, we obtain 
$$
\rho_1\ge \frac{2\sqrt{2a}}{(a-2b)}.
$$
In the same way, $ h'_x(\rho_2)\ge 0$ implies
$$
\rho_2\le \frac{2\sqrt{2a}}{(a-2b)}.
$$
As a consequence,
\begin{align*}
\rho_2\le \frac{2\sqrt{2a}}{(a-2b)}\le\rho_1\,,
\end{align*}
a contradiction.
\end{proof}

Let us next prove that the supremum of $I$ cannot be equal to $1$. For this, we use a contradiction argument. We start with some estimates for solutions of (\ref{eqradbis}) with initial data $x_n\in I$ supposing that $\lim\limits_{n\to +\infty}x_n=1$.




\begin{lem}\label{leminfgammax} Let $\{x_n\}_n\subset I$ a sequence of initial conditions such that $\lim\limits_{n\to +\infty}x_n=1$, and define $\gamma_{x_n}=\min\{r>0:f'_{x_n}(r)=0\}$. 
Then there exists a constant $c>0$ such that
\begin{equation}\label{eqinfgammax}
\inf_{n\in\mathbb N}\gamma_{x_n}=c.
\end{equation}
\end{lem}

\begin{rem}
Since $f'_{x_n}(0)<0$, $\gamma_{x_n}$ is such that 
$$
f'_{x_n}(r)<0\ \mbox{on}\ [0,\gamma_{x_n})\ \mbox{and}\ f'_{x_n}(\gamma_{x_n})=0.
$$
As a consequence, $f_{x_n}(\gamma_{x_n})<0$.
\end{rem}

\begin{proof} For convenience we denote $f_{x_n}$, $g_{x_n}$ and $\gamma_{x_n}$ by $f_n$, $g_n$ and $\gamma_n$.

Since $x_n\to 1$ as $n\to +\infty$, let us assume that $x_n$ is such that $\sqrt{\frac{2b}{a}}\le x_n^2< 1$. Then there exists $S_{x_n}\in (0,+\infty)$ such that
$$
g_{n}(S_{x_n})=x_n^2  \quad\mbox{and}\quad x_n^2\le g_{n}(r)\le x_n\quad \forall r\in [0,S_{x_n}]\,.
$$
Let denote $S_{x_n}$ by $S_n$.
Hence,
$$
x_n-x_n^2=g_{n}(0)-g_{n}(S_{n})=-\int_0^{S_{n}}{f_{n}(r)(1-g_{n}^2(r))\,dr}
=\int_0^{S_{n}}\int_0^r k_n(s)s^2\,ds\,\frac{dr}{r^2}\,,
$$
with
\begin{equation*}
k_n(s)=g_{n}(s)(1-g_{n}^2(s))(f^2_{n}(s)+a g^2_{n}(s)-b).
\end{equation*}
Next, we remark 
\begin{equation}\label{eqineqk}
x_n(1-x_n)m_{a,b}\le k_n(s)\le x_n(1-x_n)M_{a,b}\,,
\end{equation}
where $m_{a,b}$ and $M_{a,b}$ are strictly positive constants that depend only on the parameters $a$ and $b$. As a consequence,
$$
x_n(1-x_n)\frac{m_{a,b}}{6}S^2_{n}\le \int_0^{S_{n}}\int_0^r k_n(s)s^2\,ds\,\frac{dr}{r^2}\le x_n(1-x_n)\frac{M_{a,b}}{6}S^2_{n}\,,
$$
which implies
$$
\left[\frac{M_{a,b}}{6}\right]^{-1/2}\le S_{n} \le \left[\frac{m_{a,b}}{6}\right]^{-1/2}\,.
$$
Hence, there exists a constant $\tilde c>0$ such that $\inf_{n\in\mathbb N}S_{n}=\tilde c$.

Now, suppose by contradiction that $\inf_{n\in\mathbb N}\gamma_{n}=0$, then there exists a subsequence of $\{\gamma_{n}\}_n$ such that  $\gamma_{n}\xrightarrow[n]{}0$ and, for $n$ sufficiently large, $\gamma_{n}<S_{n}$. 

First of all, we remark that $f_{n}(\gamma_{n})$ satisfies  
\begin{equation}\label{eqfngamman}
\frac{2}{\gamma_{n}}f_{n}(\gamma_{n})=g_{n}(\gamma_{n})(f_{n}^2(\gamma_{n})-a g_{n}^2(\gamma_{n})+b),
\end{equation}
and it has to be strictly negative. Furthermore, since $x_n^2<g_n(\gamma_n)\le x_n$, for $n$ large enough we can write 
$$
f_{n}(\gamma_{n})=\left(\frac{1}{\gamma_{n}}-\sqrt{\frac{1}{\gamma_{n}^2}+g_{n}^2(\gamma_{n})(a g_{n}^2(\gamma_{n})-b)}\right)\frac{1}{g_{n}(\gamma_{n})}.
$$
As a consequence,
$$
\lim\limits_{n\to +\infty}f_n(\gamma_n)=\lim\limits_{n\to +\infty}\left(\frac{1}{\gamma_{n}}-\sqrt{\frac{1}{\gamma_{n}^2}+g_{n}^2(\gamma_{n})(a g_{n}^2(\gamma_{n})-b)}\right)\frac{1}{g_{n}(\gamma_{n})}=0.
$$
By the definition of $\gamma_n$ the function $f_n$ is decreasing on $[0,\gamma_n)$. Then,
$$
0>f_n(r)>f_n(\gamma_n)\quad\forall r\in(0,\gamma_n)\,,
$$
and, passing to the limit, we obtain that for all $\eps>0$ there exists $n_0(\eps)\in \mathbb{N}$ such that
$$
|f_n(r)|\le \eps\,,
$$
for all $n\ge n_0(\eps)$ and for all $r\in [0,\gamma_n]$. Moreover, using the second equation of the system (\ref{eqradbis}), we have that $g_n$ is decreasing on $[0,\gamma_n)$. Thanks to the fact that  $x_n^2<g_n(\gamma_n)\le x_n$, we get for all $\eps>0$ there exists $m_0(\eps)\in \mathbb{N}$ such that
$$
1-\eps \le g_n(r)\le1+\eps\,,
$$
for all $n\ge m_0(\eps)$ and for all $r\in [0,\gamma_n]$. Hence, for all $\eps>0$ and for $n$ sufficiently large, we have the following estimates for all $r\in (0,\gamma_n]:$
$$
(1-\eps)(b-a(1+\eps)^2) \le f'_n(r)+\frac{2}{r}f_n(r)\le(1+\eps)(b-a(1-\eps)^2+\eps^2)\,,$$
which implies
$$
\frac{(1-\eps)(b-a(1+\eps)^2)r^3}{3} \le f_n(r)r^2 \le \frac{(1+\eps)(b-a(1-\eps)^2+\eps^2)r^3}{3}\,,
$$
and, as a consequence,
\begin{equation}\label{eqestimatefn}
\frac{(1-\eps)(b-a(1+\eps)^2)}{3} \le \frac{f_n(\gamma_n)}{\gamma_n} \le \frac{(1+\eps)(b-a(1-\eps)^2+\eps^2)}{3}\,,
\end{equation}
for all $\eps>0$ and for $n$ sufficiently large.
Using (\ref{eqestimatefn}) and the first equation of the system (\ref{eqradbis}), if $\eps$ is chosen sufficiently small we obtain
\begin{align}\label{eqestimatefprime}
f'_n(\gamma_n)&=-\frac{2}{\gamma_{n}}f_{n}(\gamma_{n})+g_{n}(\gamma_{n})(f_{n}^2(\gamma_{n})-a g_{n}^2(\gamma_{n})+b)\nonumber\\
&\le -\frac{2(1-\eps)(b-a(1+\eps)^2)}{3}+(1+\eps)(b-a(1-\eps)^2+\eps^2)\le\frac{b-a}4<0\,,\nonumber
\end{align}
which contradicts the definition of $\gamma_n$. So, there exists a constant $c>0$ such that
$$
\inf_{n\in\mathbb N}\gamma_{n}=c.
$$
\end{proof}

\begin{lem}\label{lemconvgammax} Let $\{x_n\}_n\subset I$ a sequence of initial conditions such that $\lim\limits_{n\to +\infty}x_n=1$, and define $\gamma_{x_n}=\min\{r>0:f'_{x_n}(r)=0\}$. 
Then, up to a subsequence,  
\begin{equation}\label{eqconvgammax}
\lim\limits_{n\to +\infty}\gamma_{x_n}=+\infty.
\end{equation}
\end{lem}

\begin{proof} 
Thanks to Lemma \ref{leminfgammax}, we know that the sequence $\{\gamma_{x_n}\}_n$ is bounded from below by a constant $c>0$. 

Now,  suppose by contradiction that $\sup_{n\in\mathbb N}\gamma_{x_n}=C<+ \infty$. Using the continuity of the flow and the fact that the sequence $\{\gamma_{x_n}\}_n$ is bounded from above by some constant $C$, we obtain
$$
\forall \eps>0, \exists n_ \eps>0,\forall n\in\mathbb{N}, n\ge n_ \eps \Rightarrow |g_1(\gamma_{x_n})-g_{x_n}(\gamma_{x_n})|\le \eps\,,
$$
which means
\begin{equation}\label{eqconvggamma}
\lim\limits_{n\to +\infty}g_{x_n}(\gamma_{x_n})=1,
\end{equation}
since $g_1(r)=1$ for all $r>0$. With the same arguments, we get 
\begin{equation}\label{eqconvfgamma}
\lim\limits_{n\to +\infty}\left(\frac{1}{\gamma_{x_n}}-\sqrt{a-b}\,\mathrm{Coth}\left(\sqrt{a-b}\,\gamma_{x_n}\right)-f_{x_n}(\gamma_{x_n})\right)=0
\end{equation}
where $\;\frac{1}{\gamma_{x_n}}-\sqrt{a-b}\,\mathrm{Coth}\left(\sqrt{a-b}\,\gamma_{x_n}\right)$ is the expression of $f_1(\gamma_{x_n})$.
Next, we remark that $f_{x_n}(\gamma_{x_n})$ satisfies  
\be{firstnew}
\frac{2}{\gamma_{x_n}}f_{x_n}(\gamma_{x_n})=g_{x_n}(\gamma_{x_n})(f_{x_n}^2(\gamma_{x_n})-a g_{x_n}^2(\gamma_{x_n})+b)\,,
\ee
and it has to be strictly negative. Hence, for $n$ large enough, we can write 
\be{secondnew}
f_{x_n}(\gamma_{x_n})=\left(\frac{1}{\gamma_{x_n}}-\sqrt{\frac{1}{\gamma_{x_n}^2}+g_{x_n}^2(\gamma_{x_n})(a g_{x_n}^2(\gamma_{x_n})-b)}\right)\frac{1}{g_{x_n}(\gamma_{x_n})}.
\ee
So, using (\ref{eqconvfgamma}), we have 
\begin{equation}\label{eqconvfgamma2}
\lim\limits_{n\to +\infty}\left(\sqrt{\frac{1}{\gamma_{x_n}^2}+(a-b)}-\sqrt{a-b}\,\mathrm{Coth}\left(\sqrt{a-b}\,\gamma_{x_n}\right)\right)=0.
\end{equation} 
Let $h(r)$ be defined by 
\begin{equation}\label{eqdefh}
h(r):=\sqrt{\frac{1}{r^2}+(a-b)}-\sqrt{a-b}\,\mathrm{Coth}\left(\sqrt{a-b}\,r\right)\,.
\end{equation}
$h$ is a smooth function such that $\lim\limits_{r\to 0}h(r)=0=\lim\limits_{r\to +\infty}h(r)$ and $h(r)>0$ for $r\in(0,+\infty)$. Indeed, 
$$
\sqrt{\frac{1}{r^2}+(a-b)}>\sqrt{a-b}\,\mathrm{Coth}(\sqrt{a-b}\,r)
$$
if and only if 
$$
\frac{1}{r^2}+(a-b)>(a-b)\,\mathrm{Coth}^2(\sqrt{a-b}\,r)\,,
$$
or equivalently, if and only if
$$
\mathrm{Sinh}(\sqrt{a-b}\,r)>\sqrt{a-b}\,r\,.
$$
Let us analyze the function $l(x)=\mathrm{Sinh}(x)-x$. The function $l(x)$ has  the following properties:
\begin{enumerate}
\item $\lim\limits_{x\to 0}l(x)=0$; 
\item $l'(x)=\mathrm{Cosh}(x)-1>0$ on $(0,+\infty)$.
\end{enumerate}
As a conclusion, $l>0$ on $(0,+\infty)$ and $h(r)>0$ for $r\in(0,+\infty)$.
Furthermore, thanks to the fact that the sequence $\{\gamma_{x_n}\}_n$ is bounded from above, up to a subsequence,
$$
\lim\limits_{n\to +\infty}\gamma_{x_n}=\tilde\gamma
$$
with $0<c\le\tilde\gamma\le C< +\infty$; hence we obtain
$$
\lim\limits_{n\to +\infty}h(\gamma_{x_n})=h(\tilde\gamma)>0\,,
$$
which contradicts the limit in (\ref{eqconvfgamma2}).
As a consequence, $\sup_{n\in\mathbb N}\gamma_{x_n}=+ \infty$ and there exists a subsequence of $\{\gamma_{x_n}\}_n$ such that
$$
\lim\limits_{n\to +\infty}\gamma_{x_n}=+\infty.
$$
\end{proof}
To summarize, let $\{x_n\}_n\subset I$ be a sequence of initial conditions such that $\lim\limits_{n\to +\infty}x_n=1$, $\gamma_{x_n}=\min\{r>0:f'_{x_n}(r)=0\}$ and $S_{x_n}=\min\{r>0:g_{x_n}(r)=x_n^2\}$ as in the proof of Lemma \ref{leminfgammax}. Then, up to a subsequence, $
\lim\limits_{n\to +\infty}\gamma_{x_n}=+\infty
$, $\inf\limits_{n\in\mathbb{N}}S_{x_n}>0$ and $\sup\limits_{n\in\mathbb{N}}S_{x_n}<+\infty$. Moreover $\gamma_{x_n}<r_{x_n}$.

Let us now end Step 2 by proving the following lemma.

\begin{lem}\label{lemsupsmaller1} Let $\tilde x=\sup I$. Then $\tilde x<1$.
\end{lem}

\begin{proof} For convenience we denote $f_{x_n}$, $g_{x_n}$, $\gamma_{x_n}$ and $S_{x_n}$ by $f_n$, $g_n$, $\gamma_n$ and $S_n$.
We already proved that $\sup I\le 1$. Then suppose, by contradiction, that there exists a sequence $\{x_n\}_n\subset I$ such that $\lim\limits_{n\to +\infty}x_n=1$. 

We claim then that in  this case there exists a sequence $\{t_n\}_n$ such that $S_n\le t_n\le \gamma_n$ for all $n\in \mathbb{N}$ and $\lim\limits_{n\to +\infty}|f'_{n}(t_n)|+|g'_{n}(t_n)|=0$.

Indeed, suppose by contradiction that, up to a subsequence, $$\lim\limits_{n\to +\infty}\inf\limits_{r\in[S_n,\gamma_n]}|f'_{n}(r)|+|g'_{n}(r)|>0\,.$$  

We define $N_{x}(R_1,R_2)$ as the number of roots of $g_x$ in $\{r\in(R_1,R_2) \mbox{ ; } g_x^2(r)+f_x^2(r)\neq 0\}$ and use the notation $N_x(R)=N_x(0,R)$. We remark that the number of roots $N_x(R_1,R_2)$ is linked to the winding number of the trajectory around $(0,0)$ in the following way:
\begin{equation}\label{eqnumroot}
N_x(R_1,R_2)=\frac{\theta_x(R_2)-\theta_x(R_1)}{\pi}\,,
\end{equation}
where $\theta_x(r)=-\arctan\left(\frac{f_x(r)}{g_x(r)}\right)$.
Then, we compute the derivative of $\theta_x(r)$ and we get
$$
\theta'_x(r)=\frac{f_x(r)g'_x(r)-f'_x(r)g_x(r)}{f^2_x(r)+g^2_x(r)}.
$$
We note that $\theta'_x(0)=\frac{-f'_x(0)}{x}>0$. Hence $\theta_x(r)$ is increasing in a neighborhood of $0$. Let denote $N_{x_n}$ and $\theta_{x_n}$ by $N_n$ and $\theta_n$. 

First of all, we remind that for all $n\in \mathbb{N}$ the following properties hold for all $r\in[S_n,\gamma_n]$:
\begin{itemize}
\item $f'_n(r)\le0$;
\item $-\sqrt{a-b}< f_n(\gamma_n)<f_n(r)<f_n(S_n)\le -\delta<0$ for some $\delta>0$ that does not depend on $n$;
\item $g'_n(r)<0$;
\item $\sqrt{\frac{b}{a}}<g_n(\gamma_n)<g_n(r)<g_n(S_n)=x_n^2<1$.
\end{itemize}
Note that the first inequality of the last property is obtained using the first equation of (\ref{eqradbis}) and the definition of $\gamma_n$. Moreover, the inequality $f_n(S_n)\le -\delta<0$ will be proved below.

Using these properties and the definition of $N_n(S_n,\gamma_n)$, we obtain 
\begin{equation}\label{eqnumrootSngamman}
N_n(S_n,\gamma_n)=0
\end{equation}
for all $n\in \mathbb{N}$.

On the other hand, as we prove below, we have
\begin{equation}\label{eqnumrootconv}
\lim\limits_{n\to+\infty}N_n(S_n,\gamma_n)=+\infty.
\end{equation}
Indeed,
\begin{align*}
 \liminf\limits_{n\to+\infty}N_n(S_n,\gamma_n)&=\liminf\limits_{n\to+\infty}\frac{1}{\pi}\int_{S_n}^{\gamma_n}\theta'_n(r)\,dr\\&\ge \frac{1}{\pi}\liminf\limits_{n\to+\infty} \inf\limits_{r\in[S_n,\gamma_n]}\theta'_n(r)\left(  \lim\limits_{n\to+\infty}\gamma_n+\limsup\limits_{n\to+\infty}S_n\right).
\end{align*}
Hence, if $\liminf\limits_{n\to+\infty} \inf\limits_{r\in[S_n,\gamma_n]}\theta'_n(r)>0$, we obtain $$\liminf\limits_{n\to+\infty}N_n(\gamma_n)=+\infty$$ since $ \lim\limits_{n\to+\infty}\gamma_n=+\infty$. Therefore, it remains to prove 
\begin{equation}\label{eqliminftheta}
\liminf\limits_{n\to+\infty} \inf\limits_{r\in[S_n,\gamma_n]}\theta'_n(r)>0.
\end{equation}
We remark that, for all $r\in[S_n,\gamma_n]$,
\begin{align*}
\theta'_n(r)&=\frac{f_n(r)g'_n(r)-f'_n(r)g_n(r)}{f^2_n(r)+g^2_n(r)}=\frac{|f_n(r)||g'_n(r)|+|f'_n(r)||g_n(r)|}{f^2_n(r)+g^2_n(r)}\\
&\qquad\ge\frac{\min\left\{\delta,\sqrt{\frac{b}{a}}\right\}}{a-b+1}(|g'_n(r)|+|f'_n(r)|),
\end{align*}
which implies
$$
 \inf\limits_{r\in[S_n,\gamma_n]}\theta'_n(r)\ge C\inf\limits_{r\in[S_n,\gamma_n]}|g'_n(r)|+|f'_n(r)|\,,
$$
with $C=\frac{\min\left\{\delta,\sqrt{\frac{b}{a}}\right\}}{a-b+1} >0$. Hence, 
$$
 \liminf\limits_{n\to+\infty}\inf\limits_{r\in[S_n,\gamma_n]}\theta'_n(r)\ge C  \lim\limits_{n\to+\infty}\inf\limits_{r\in[S_n,\gamma_n]}|g'_n(r)|+|f'_n(r)|>0.
$$
As a conclusion, if up to a subsequence  $\lim\limits_{n\to +\infty}\inf\limits_{r\in[S_n,\gamma_n]}|f'_{n}(r)|+|g'_{n}(r)|>0$, (\ref{eqnumrootSngamman}) and (\ref{eqnumrootconv}) provide a contradiction.\\

Hence, there exists a sequence $\{t_n\}_n$ such that for all $n\in \mathbb{N}$,  $S_n\le t_n\le \gamma_n$  and $$\lim\limits_{n\to +\infty}|f'_{n}(t_n)|+|g'_{n}(t_n)|=0.$$

With the same arguments used in the proof of Lemma \ref{lemconvgammax}, we obtain, up to a subsequence,  
$$
\lim\limits_{n\to +\infty}t_n=+\infty.
$$ 
Indeed, the only difference between this case and the previous one lies in the fact that $f'(\gamma_n)=0$ while $f'(t_n)$ tends only towards $0$ as $n$ goes to $+\infty$. Then, the only change to be made in order to adapt the arguments used in the proof of Lemma \ref{lemconvgammax} to the present case consists in adding a small $o(1)_{n\to +\infty}$ to the r.h.s. of both \eqref{firstnew} and \eqref{secondnew}.

Next, we remark that 
\begin{equation}\label{eqlimsupfntn}
\limsup\limits_{n\to+\infty}f_n(t_n)\le \limsup\limits_{n\to+\infty}f_n(S_n)<0\,.
\end{equation}
Indeed, since $S_n\le t_n\le \gamma_n$ for all $n\in \mathbb{N}$  and by the definition of $\gamma_n$, we have $f_n(\gamma_n)<f_n(t_n)<f_n(S_n)$ for all $n\in \mathbb{N}$. It remains to prove that $\limsup\limits_{n\to+\infty}f_n(S_n)<0$. Using the second equation of the system (\ref{eqradbis}), we have
$$
S_n f_n(S_n) \le \int_0^{S_n}f_n(r)\,dr= \int_0^{S_n}\frac{g'_n(r)}{1-g^2_n(r)}\,dr=\frac{1}{2}\log\frac{1+x_n^2}{(1+x_n)^2}\le \delta <0  \,,
$$
since $\frac{1+x_n^2}{(1+x_n)^2}<1\,$ uniformly. As a consequence, and using the fact that $\{S_n\}$ is bounded, we find
$$
\limsup\limits_{n\to+\infty}f_n(S_n)<0.
$$ 
Therefore, by the equation satisfied by $g_n$, $\lim\limits_{n\to +\infty}g'_{n}(t_n)=0\;$ implies $\,\lim\limits_{n\to +\infty}g_{n}(t_n)=1$. Hence, using the first equation of (\ref{eqradbis}) and the properties of convergence of $\{f'_n(t_n)\}_n$ and $\{t_n\}_n$, we obtain, up to a subsequence, $\lim\limits_{n\to +\infty}f^2_{n}(t_n)=a-b$. As a conclusion, $\lim\limits_{n\to +\infty}f_{n}(t_n)=-\sqrt{a-b}\,$ since $\,f_n(t_n)<0$.

Next, let us set $(w_n, z_n)=(f_n(t_n+\cdot),g_n(t_n+\cdot))$. We have 
\begin{align*}
&\lim\limits_{n\to+\infty}w_n(0)=-\sqrt{a-b}\,,\\
&\lim\limits_{n\to+\infty}z_n(0)=1.
\end{align*}
Let $(w,z)$ be the solution of (\ref{eqradconsbis}) with initial data $(-\sqrt{a-b}\,,1)$. Since $(-\sqrt{a-b}\,,1)$ is a rest point of (\ref{eqradcons}), $w(r)=-\sqrt{a-b}$ and $z(r)=1$ for all $r\in[0,+\infty)$. Hence, thanks to the condition $a-2b>0$, we obtain
\begin{equation}\label{eqsolconsinequality}
-w(r)>\sqrt{\frac{a}{2}}z(r)\,,
\end{equation}
for all $r\in [0,+\infty)$.
On the other hand, applying Lemma \ref{lemconvconssol}, and since $\lim\limits_{n\to+\infty}t_n=+\infty$, we have that for all $r\in(0,+\infty)$, $(w_n,z_n)$ converges to $(w,z)$ uniformly on $[0,r]$. Therefore for $n$ large enough and for all $r\in(0,+\infty)$, we have 
$$
f_n(t_n+r)+\sqrt{\frac{a}{2}}g_n(t_n+r)<0 \,,
$$
which contradicts Lemma \ref{lempropertiesI2}. This proves the lemma.
\end{proof}

\noindent\emph{Step $3$.} Conclusion and proof of Theorem \ref{thmexistence}.\\

Let $(\tilde f, \tilde g)=(f_{\tilde x},g_{\tilde x})$ with $\tilde x= \sup I $. Thanks to Lemma \ref{lemsupsmaller1} and by the definition of $I$, $\sqrt{\frac{b}{a}}<\tilde x < 1$. Hence, by Lemma \ref{lempropertiessol}, we have $R_{\tilde x}=+\infty$.
Since $\tilde x>\sqrt{\frac{b}{a}}$, we have $\tilde f'(0)<0$, and thus $\tilde f<0$ and $\tilde g>0$ in a neighborhood of $0$.
Moreover, since the set $I$ is open, $\tilde x \notin I$ and $\tilde f$ cannot vanish before $\tilde g$ does. 

On the other hand, $\tilde x \in \bar I$. Then,  $\tilde g$ cannot vanish before $\tilde f$ because of the continuous dependence of $(f_x,g_x)$ with respect to $x$. 
Finally, $\tilde f$ and $\tilde g$ cannot vanish simultaneously because $(0,0)$ is a rest point of (\ref{eqradbis}).
Therefore, $\tilde f(r)<0$ and $\tilde g(r)>0$ for all $r\in(0,+\infty)$. Then, using Lemma \ref{lemexpdecay},  there exists a constant $C$ such that for all $r\in(0,+\infty)$,
$$
0\le -f_x(r), g_x(r)\le C\exp(-K_{a,b}r)\,,
$$
with $K_{a,b}=\min\left\{\frac{b}{2},\frac{2a-b}{2a}\right\}$.
${}\hfill\square$

\section{Numerical aspects and qualitative properties of ground states}\label{numerics}

As it has been observed in computations carried out in physics, the fields of mesons $\sigma$ and $\omega$ are proportional to the scalar and the vector density. In models for finite nuclei these fields approximately assume a ``plateau"-like Saxon--Woods shape: they vanish outside the nucleus and they are more or less constant inside it.  Moreover, the intensity of the potential for the antinucleons $V-S$ is much more higher than that of the potential for the nucleons $V+S$ (see for instance \cite{ring}).

We have run some numerical calculations for our model, trying to see how the values of the parameters in the problem affect the shapes of those fields and the intensity of the potentials. We observe that the results depend a lot on the values of $a$ and $b$, which are related to the physical values of the meson masses and of the coupling constants. More precisely, we remark that even in this very particular case where only one nucleon is taken into account, the Saxon--Woods shape for the potentials $V$ and $S$ is perfectly observable and the magnitude of $|V-S|$ is much larger than that of $|V+S|$, if the parameters $a$ and $b$ are well chosen.




\begin{figure}[h!]
   \begin{minipage}[c]{.47\linewidth}
      \includegraphics{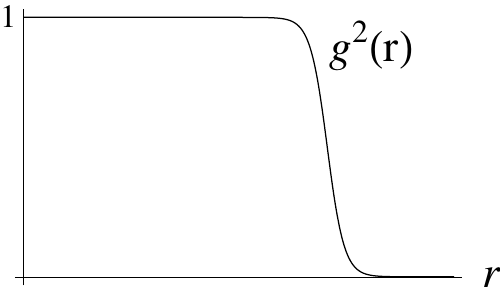}
      \caption{Plot of $g^2$, $\;\frac{2b}{a}\sim 1$}
	\label{Plotg2-1}
   \end{minipage}
    \begin{minipage}[c]{.47\linewidth}
      \includegraphics{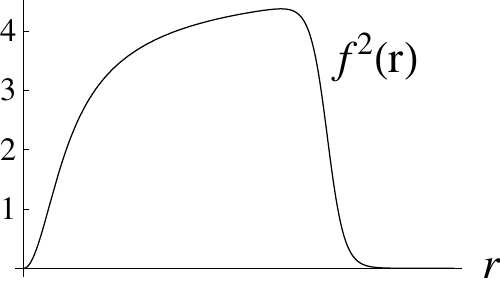}
      \caption{Plot of $f^2$, $\;\frac{2b}{a}\sim 1$}
	\label{Plotf2-1}
   \end{minipage}
   \end{figure}
   
In Figures \ref{Plotg2-1} and \ref{Plotf2-1}, we plot the shapes of $f^{2}$ and $g^{2}$ for the values $a=9$ and $b=4$.
   
The ``plateau"-like Saxon-Woods shape is perfectly clear for $g^{2}$. For $f^{2}$ it is not, but taking into account that as $c$ goes to $+\infty$,  $S \sim -mc^2g^2 +f^2/(4m)$  and $V\sim mc^2g^2-a g^2/(2m)+ f^2/(4m)$, the properties discussed above for $V$, $S$, $V+S$ and $V-S$ are verified here in a very straightforward way. Indeed, the above asymptotics show that $V$ and $S$ behave like a plateau if $g^{2}$ does, and the intensity of $|V-S|$ is much higher than that of $|V+S|$.

A less clear case is the next one, see figures  \ref{Plotg2-2} and \ref{Plotf2-2}, where the values of $a$ and $b$ are respectively of $4$  and $1$. Here the ``plateau" is much less visible and its edge are much less sharp. Practically there is no ``plateau" in this case.

\begin{figure}[h!]
   \begin{minipage}[c]{.47\linewidth}
      \includegraphics{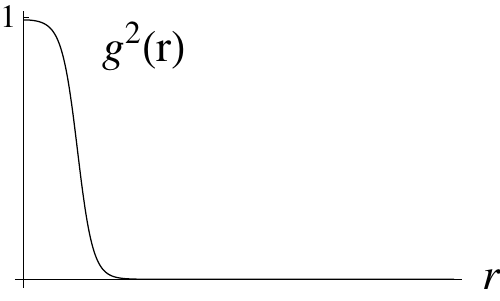}
      \caption{Plot of $g^2$, $\;\frac{2b}{a}\not\sim 1$}
	\label{Plotg2-2}
   \end{minipage}
    \begin{minipage}[c]{.47\linewidth}
      \includegraphics{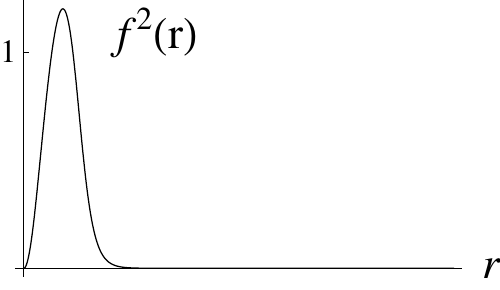}
      \caption{Plot of $f^2$, $\;\frac{2b}{a}\not\sim 1$}
	\label{Plotf2-2}
   \end{minipage}
\end{figure}

Actually, other computations that we have run show that the ``plateau" shape is more and more present when $2b/a$ approaches $1$.

\subsection*{Acknowledgment}
This work was partially supported by the Grant ANR-10-BLAN 0101 (NONAP) of the French Ministry of Research.


\begin{thebibliography}{10}
\expandafter\ifx\csname urlstyle\endcsname\relax
  \providecommand{\doi}[1]{doi:\discretionary{}{}{}#1}\else
  \providecommand{\doi}{doi:\discretionary{}{}{}\begingroup
  \urlstyle{rm}\Url}\fi
\providecommand{\selectlanguage}[1]{\relax}

\bibitem{cazenavevazquez}
Cazenave, T., Vazquez, L.
\newblock Existence of {L}ocalized {S}olutions for a {C}lassical {N}onlinear
  {D}irac {F}ield.
\newblock \emph{Commun. Math. Phys.}, 105, 35--47, 1986.

\bibitem{gognylions}
Gogny, D., Lions, P.L.
\newblock Hartree-{F}ock theory in nuclear physics.
\newblock \emph{RAIRO Mod{\'e}l. Math. Anal. Num{\'e}r.}, 20(4), 571--637,
  1986.

\bibitem{greinermaruhn}
Greiner, W., Maruhn, J.
\newblock \emph{Nuclear Models}.
\newblock Springer-Verlag, Berlin, 1996.

\bibitem{meng}
Meng, J., Toki, H., Zhou, S., Zhang, S., Long, W., Geng, L.
\newblock Relativistic continuum {H}artree {B}ogoliubov theory for ground-state
  properties of exotic nuclei.
\newblock \emph{Prog. Part. Nucl. Phys.}, 57, 470--563, 2006.

\bibitem{reinhard}
Reinhard, P.G.
\newblock The relativistic mean-field description of nuclei and nuclear
  dynamics.
\newblock \emph{Rep. Prog. Phys.}, 52, 439--514, 1989.

\bibitem{ring}
Ring, P.
\newblock Relativistic {M}ean {F}ield {T}heory in {F}inite {N}uclei.
\newblock \emph{Prog. Part. Nucl. Phys.}, 37, 193--236, 1996.

\bibitem{Ring-Schuck}
Ring, P., Schuck, P.
\newblock \emph{The Nuclear Many-Body Problem}.
\newblock Springer-Verlag, Heidelberg, 1980.

\bibitem{rotanodarirmf}
Rota~Nodari, S.
\newblock The Relativistic Mean-Field Equations of the Atomic Nucleus.
\newblock \emph{Rev. Math. Phys.}, 24(4), 125008, 2012.

\bibitem{Thaller}
Thaller, B.
\newblock \emph{The Dirac equation}.
\newblock Springer-Verlag, 1992.

\bibitem{waleckasigmaomega}
Walecka, J.D.
\newblock A theory of highly condensed matter.
\newblock \emph{Ann. Physics}, 83(2), 491 -- 529, 1974.

\bibitem{walecka}
Walecka, J.D.
\newblock \emph{Theoretical Nuclear and Subnuclear Physics}.
\newblock Imperial College Press and World Scientific Publishing Co. Pte. Ltd.,
  second edition edition, 2004.

\end{thebibliography}

\end{document}